\documentclass[notitlepage,a4paper,twoside,leqno,
11pt,
]{revtex4-1}

\usepackage{hyperref}
\hypersetup{colorlinks=true,allcolors=[rgb]{0,0,0.6}}
\usepackage[utf8]{inputenc}
\usepackage{color}
\usepackage{calc}
\usepackage{newlfont}
\usepackage[english]{babel}
\usepackage{amssymb,amsmath,amsfonts,amsthm}
\usepackage{mathrsfs}
\usepackage{dsfont}
\usepackage{enumerate}
\usepackage[shadow,textwidth=12mm,textsize=tiny]{todonotes}
\DeclareMathOperator{\Z}{\mathscr{Z}}
\DeclareMathOperator{\rz}{\mathds{R}}
\DeclareMathOperator{\nz}{\mathds{N}}

\DeclareMathOperator{\Tr}{Tr}

\renewcommand{\Im}{\mathrm{Im}}
\renewcommand{\Re}{\mathrm{Re}}

\theoremstyle{plain}
\newtheorem{thm}{Theorem}[section]
\newtheorem{proposition}[thm]{Proposition}
\newtheorem{lemma}[thm]{Lemma}
\newtheorem{corollary}[thm]{Corollary}
\theoremstyle{definition}
\newtheorem{definition}[thm]{Definition}

\theoremstyle{remark}
\newtheorem*{example}{Example}
\newtheorem{remark}[thm]{Remark}
\newtheorem*{remark*}{Remark}

\begin{document}
\title{Wigner measures approach to the classical limit of the Nelson
  model: Convergence of dynamics and ground state energy.}

\date{\today} \author{Zied Ammari}
\email[]{zied.ammari@univ-rennes1.fr} \author{Marco Falconi}
\email[]{marco.falconi@univ-rennes1.fr}
\affiliation{IRMAR and Centre Henri Lebesgue; Université de Rennes I\\
  Campus de Beaulieu, 263 Avenue du Général Leclerc\\
  CS 74205, 35042 Rennes Cedex}
\begin{abstract}
  We consider the classical limit of the Nelson model, a system of
  stable nucleons interacting with a meson field. We prove convergence
  of the quantum dynamics towards the evolution of the coupled
  Klein-Gordon-Schrödinger equation. Also, we show that the ground
  state energy level of $N$ nucleons, when $N$ is large and the meson
  field approaches its classical value, is given by the infimum of the
  classical energy functional at a fixed density of particles. Our
  study relies on a recently elaborated approach for mean field theory
  and uses Wigner measures.
\end{abstract}
\maketitle

\section{Introduction}
\label{sec:introduction}

The Nelson model refers to a quantum dynamical system describing a
nucleon field interacting with a meson scalar Bose field.  When an
ultraviolet cut off is put in the interaction, the Hamiltonian becomes
a self-adjoint operator and so the quantum dynamics is well defined.
In the early sixties, E.~Nelson showed that the quantum dynamics of this
system exists even when the ultraviolet cut off is removed
\citep[see][]{N}. It is indeed one of the simplest examples in
non-relativistic quantum field theory (QFT) where renormalization is needed and successfully performed using only basic tools of functional analysis.

Over the past two decades, there has been considerable effort devoted
to the study of the Nelson model that have led to a thorough
investigation of its spectral and scattering properties \citep[see][to
mention but a few]{Ar01,AH12,BFS2,BFLMS,DG1,FGS3,GGM1,GHPS,Mol,Pi,Sp3}.  However,
the fact that the Nelson Hamiltonian is  a  Wick
quantization of a classical Hamiltonian  system is quite often neglected
except in few references \cite{Fa,GNV}.  We believe that the study of
the classical limit of such quantum dynamical systems is  a significant question
leading  to an unexplored phase-space point of view in QFT. This for sure will enrich
the subject and may also provide some insight on some of the remaining
open problems.

In this article, we neglect the spin and isospin for nucleons, so
we are considering a scalar Yukawa  field theory. We also suppose that an
ultraviolet cut off is imposed on the interaction.  We prove two main
results stated in
Theorem \ref{main.th.1} and Theorem \ref{main.th.2}, namely:
\begin{enumerate}[(i)]
\item\label{item:6} Convergence of the quantum dynamics towards the classical
evolution of the coupled Klein-Gordon-Schrödinger equation.
\item\label{item:7} Convergence of the ground state energy level of $N$ nucleons to
the infimum of the classical energy functional with fixed density of
particles,
when $N$ tends to infinity and  the Bose field approaches its classical limit.
\end{enumerate}
There are basically two schemes for proving~\eqref{item:6}: either one studies the propagation of states or those of observables. The latter strategy being very difficult for systems with unconserved number of particles, we rely on the first scheme.
To establish~\eqref{item:6}, we follow indeed a Wigner measures approach, recently
elaborated in \cite{AmNi1,AmNi2,AmNi3,2011arXiv1111.5918A} for the purpose of mean
field limit in many-body theory.  This method turns out to be quite general
and flexible. It can be adapted to quantum electrodynamics (QED) and
relativistic quantum field theory (QFT) and it  gives a fair description of the propagation of general states in the classical limit (see Theorem \ref{main.th.1}).
Actually, the convergence~\eqref{item:6} is known in the particular case of coherent states  \citep[see][]{Fa,GNV} by Hepp's method \cite{Hep} which relies on the special   structure of those states. The result in Theorem \ref{main.th.1} says that the
convergence of the Neslon quantum  dynamics towards the classical one has nothing to do with any particular structure or choice of states but it is rather a general (Bohr) quantum-classical correspondence principle for a system with an infinite number of degrees of freedom. In this sense, Theorem \ref{main.th.1} is more general and provides a better understanding of the classical limit of Nelson Hamiltonians.

In addition to the fact that the Wigner measures approach gives a stronger convergence result compared to the coherent state method, it also proves to be a powerful tool for tackling variational questions of type~\eqref{item:7}. Indeed, asymptotic properties of a given minimizing sequence can be derived by looking to its Wigner measures and it turns out that some
 a priori information on these Wigner measures are crucial in the proof of Theorem \ref{main.th.2}. However, both our results give only the limit of quantum quantities in terms of their classical approximations and provide no error bound on the difference.   This is of course an interesting question, among several others, and it is beyond the scope of this article. Actually, our work is also meant to stimulate further investigations and 
 to underline some open problems. For instance, removal of the momentum cutoff and  drop of the confining potential as well as time asymptotics and scattering theory  within the classical limit are quite interesting open questions. We believe indeed that our work provides a  basis for further developments on the above-mentioned problems.

The phase space of the theory is $\mathscr{Z}:=L^2(\mathds{R}^d)\oplus
L^2(\mathds{R}^d)$, and we consider the symmetric Fock space
$\mathscr{H}:=\Gamma_s(\mathscr{Z})\sim
\Gamma_s(L^2(\mathds{R}^d))\otimes\Gamma_s(L^2(\mathds{R}^d))$.  We
denote by $\psi^{\#}$ the annihilation and creation of the
non-relativistic particles (nucleons), by $a^{\#}$ the annihilation
and creation of the relativistic meson field. We recall that for each
$\varepsilon\in(0,\bar \varepsilon)$, with $\bar\varepsilon>0$ fixed
once and for all, we choose the algebra:
\begin{equation*}
  [\psi(x_1),\psi^*(x_2)]=\varepsilon \delta(x_1-x_2)\;,\quad [a(k_1),a^*(k_2)]=\varepsilon \delta(k_1-k_2)\;.
\end{equation*}
This fixes the scaling so that each $\psi^{\#}$ and $a^{\#}$ behaves
like $\sqrt{\varepsilon}$. For instance, the second quantization
operators $d\Gamma(\cdot)=\int_{\mathds{R}^{d}} a^{*}(k) (\,\cdot\,)
a(k) dk $ or $\int_{\mathds{R}^{d}} \psi^{*}(x) (\,\cdot\,) \psi(x) dx
$ scale like $\varepsilon$.  This is also the case for the number
operators $N_{1}=d\Gamma(1)\otimes 1, N_{2}=1\otimes d\Gamma(1)$ and
$N=N_{1}+N_{2}$.  The Weyl operators are $W(\xi)=W(\xi_{1})\otimes
W(\xi_{2})$, for $\xi=\xi_{1}\oplus\xi_{2}\in\Z$, with
$W(\xi_{1})=e^{i \frac{\psi^{*}(\xi_{1})+\psi(\xi_{1})}{\sqrt{2}}}$
and $W(\xi_{2})=e^{i \frac{a^{*}(\xi_{2})+a(\xi_{2})}{\sqrt{2}}}$
being the Weyl
operators on $\Gamma_s(L^2(\mathds{R}^d))$.\\
In the Fock representation of these canonical commutation relations,
the Nelson Hamiltonian takes the form:
\begin{equation*}
  \begin{split}
    H=d\Gamma(-\frac{\Delta}{2M}+V)\otimes 1+1\otimes d\Gamma(\omega)+\int_{\mathds{R}^{2d}}^{}\frac{\chi(k)}{\sqrt{\omega(k)}}\psi^{*}(x)\bigl(a^{*}(k)e^{-ik\cdot x}+a(k)e^{ik\cdot x}\bigr)\psi(x) dkdx\; ;
  \end{split}
\end{equation*}
where $\omega(k)=\sqrt{k^2+m_0^2}$ and $m_0\geq 0$. Here $m_0$ and $M$
are respectively the meson and nucleon mass at rest. It is useful to
split $H$ in a free part $H_0$, and an interaction part $H_I$, with:
\begin{align*}
  H_0&=d\Gamma(-\frac{\Delta}{2M}+V)\otimes 1+1\otimes d\Gamma(\omega)\; ,\\
  H_I&=\int_{\mathds{R}^{2d}}^{}\frac{\chi(k)}{\sqrt{\omega(k)}}\psi^{*}(x)\bigl(a^{*}(k)e^{-ik\cdot x}+a(k)e^{ik\cdot x}\bigr)\psi(x) dkdx\; .
\end{align*}
We assume the potential $V(x)$ to be in
$L^2_{loc}(\mathds{R}^{d},\mathds{R}_+)$, so that $-\Delta +V$ is a
positive self-adjoint operator on $L^2(\mathds{R}^d)$, by Kato
inequality, and essentially self-adjoint on
$C_0^\infty(\mathds{R}^d)$. The main assumption we require on the cut
off function $\chi$ is that $\omega^{-1/2}\chi\in
L^2(\mathds{R}^{d})$. This is enough to define $H$ as self-adjoint
operator (see Proposition \ref{prop:14}). To recapitulate, we assume
throughout the article the assumption
\begin{equation}
  \label{eq:5}
  \tag{{\bf A}}
  V\in L^2_{loc}(\mathds{R}^{d},\mathds{R}_+)\text{ and } \omega^{-1/2}\chi\in L^2(\mathds{R}^{d})\,.
\end{equation}
Actually, the Nelson Hamiltonian is a Wick quantization of the
classical energy functional
\begin{equation*}
  h(z_1\oplus z_2)=\langle z_1,-\frac{\Delta}{2M}+V \,z_1\rangle+\langle z_2,\omega(k) z_2\rangle+ \int_{\mathds{R}^{2d}}^{}\frac{\chi(k)}{\sqrt{\omega(k)}} |z_1|^2(x) \bigl(\bar{z}_2(k) e^{-ik\cdot x}+z_2(k)e^{ik\cdot x}\bigr) dkdx\;.
\end{equation*}
The Hamiltonian $h$ describes the coupled Klein-Gordon-Schrödinger
system with an Yukawa type interaction subject to a momentum cut
off. With the assumption \eqref{eq:5}, the related Cauchy problem is
well posed in $\mathscr{Z}$ (see Propositions \ref{lemma:1} and
\ref{prop:3}).

The main point in the proof of (i) is to understand the propagation of
normal states on the Fock space $\mathscr{H}$ with the appropriate
scaling. The idea is to encode the oscillations of any family of
states with respect to the semiclassical parameter $\varepsilon$ by
classical quantities, namely probability measures on the phase space
(Wigner measures).  Then (i) can be restated as the propagation of
these measures along the classical flow of the
Klein-Gordon-Schrödinger equation.

We say that a Borel probability measure $\mu$ on $\Z$ is a Wigner
measure of a family of normal states
$(\varrho_{\varepsilon})_{\varepsilon\in (0,\bar \varepsilon)}$ on
$\mathscr{H}$ if there exists a sequence
$(\varepsilon_k)_{k\in\mathds{N}}$ in $(0,\bar\varepsilon)$, such that
$\varepsilon_k\to 0$ and for any $\xi\in\mathscr{Z}$,
\begin{equation}
\label{quan1}
\lim_{k\to\infty} \Tr[\varrho_{\varepsilon_k} W(\xi)]= \int_{\Z}
e^{i\sqrt{2} {\Re}\langle\xi, z\rangle} \, d\mu(z)\,,
\end{equation}
where $W(\xi)$ refers to the Weyl operator on the Fock space
$\mathscr{H}$ which depends on $\varepsilon_{k}$ (here $\Re \langle
\cdot , \cdot\rangle_{}$ is the real part of the scalar product on
$\mathscr{Z}$).  We denote the set of all Wigner measures of a given
family of states $(\varrho_{\varepsilon})_{\varepsilon\in (0,\bar
  \varepsilon)}$ by $\mathscr{M}(\varrho_{\varepsilon}, \varepsilon\in
(0,\bar \varepsilon))$.  It was proved in \cite{AmNi1} that the
assumption
\begin{equation*}
  \exists \delta>0, \exists C>0, \forall \varepsilon\in (0,\bar\varepsilon) \quad  \Tr[\varrho_{\varepsilon} N^{\delta}]<C\,,
\end{equation*}
ensures that the set of Wigner measures
$\mathscr{M}(\varrho_{\varepsilon}, \varepsilon\in (0,\bar
\varepsilon))$ is non empty. Notice that the $\Tr[\,\cdot\,]$ is understood as
$\sum_{i=0}^{\infty}\lambda_i\langle \varphi_i , N^{\delta}\varphi_i
\rangle_{}$, where $\{\varphi_i\}_{i\in\mathds{N}}$ is an O.N.B. of
eigenvectors of $\varrho_{\varepsilon}$ associated to the eigenvalues $\{\lambda_i\}_{i\in\mathds{N}}$.

The Nelson Hamiltonian $H$ has a fibred structure with respect to the number of nucleons. So, it can be written as $H=\oplus_{n=0}^{\infty} H_{|L^{2}_s(\rz^{nd})\otimes\Gamma_s(L^{2}(\rz^{d}))}$ where
$L_s^2(\rz^{nd})$ denotes the space of symmetric square integrable functions (see Section
\ref{sec:dynam-quant-class}). It also turns out that $H$ is unbounded from below while $H_{|L^{2}_s(\rz^{nd})\otimes\Gamma_s(L^{2}(\rz^{d}))}$ is bounded from below.
\\
Under the aforementioned assumptions on the potential $V$ and the cut
off function $\chi$, we are in position to precisely state  our two
main results.
\begin{thm}
  \label{main.th.1}
  Assume that \eqref{eq:5} holds. Let
  $(\varrho_{\varepsilon})_{\varepsilon\in (0,\bar \varepsilon)}$ be a
  family of normal states on the Hilbert space $\mathscr{H}$
  satisfying the assumption:
  \begin{equation}
    \label{eq:0}
    \exists \delta>0, \exists C>0, \forall \varepsilon\in (0,\bar\varepsilon) \quad \Tr[\varrho_{\varepsilon} (H_0+N+1)^{\delta}]<C.
  \end{equation} 
  Then for any $t\in\mathds{R}$ the set of Wigner measures associated
  with the family $(e^{-i \frac{t}{\varepsilon}
    H}\varrho_{\varepsilon}e^{i \frac{t}{\varepsilon} H})_{\varepsilon\in (0,\bar
    \varepsilon)}$ is
  $$
  \mathscr{M}(e^{-i \frac{t}{\varepsilon} H}\varrho_{\varepsilon}e^{i \frac{t}{\varepsilon} H}, \varepsilon\in (0,\bar \varepsilon))=\{\Phi(t,0)_{\#}\mu_{0}, \mu_{0}\in \mathscr{M}(\varrho_{\varepsilon}, \varepsilon\in (0,\bar \varepsilon))\}\,,
  $$
  where $\Phi(t,0)_{\#}\mu_{0}$ is the push forward of $\mu_{0}$ by
  the classical flow $\Phi(t,s)$ of the coupled
  Klein-Gordon-Schrödinger equation \eqref{eq:1} well defined and
  continuous on $\Z$ by Propositions \ref{lemma:1} and \ref{prop:3}.
\end{thm}

\begin{thm}
  \label{main.th.2}
  Assume that \eqref{eq:5} holds and additionally $m_0>0$ and $V$ is a
  confining potential, i.e.: $\lim_{|x|\to\infty} V(x)=+\infty $. Then
  the ground state energy of the restricted Nelson Hamiltonian has the
  following limit, for any $\lambda>0$,
  \begin{equation}
    \displaystyle\lim_{\varepsilon\to 0, n\varepsilon=\lambda^2}
    \inf \sigma(H_{|L^{2}_s(\mathds{R}^{dn})\otimes
      \Gamma_{s}(L^{2}(\mathds{R}^{d}))})= \displaystyle\inf_{||z_{1}||_{L^2(\mathds{R}^d)}=\lambda} h(z_{1}\oplus z_{2})\,,
  \end{equation}
  where the infimum on the right hand side is taken over all $z_1\in D(\sqrt{\frac{-\Delta}{2M}+V})$ and $z_2\in
  D(\omega^{1/2})$ with the constraint $||z_{1}||_{L^2(\mathds{R}^d)}=\lambda$.
\end{thm}

The proof of Theorem \ref{main.th.1} is given in Sections
\ref{sec:trace-states} and \ref{sec:mean-field-limit} and uses the properties of the quantum and classical dynamics proved in Section \ref{sec:dynam-quant-class}.
It is rather lengthy, so for reader's convenience we outline its key arguments below.
The proof of Theorem \ref{main.th.2}, given in Section \ref{sec:gdstate}, relies on an upper bound derived
by using coherent states localized around the infimum of the classical
energy and a lower bound resulting from the a priori information on
the Wigner measures of a given minimizing sequence. So, we conclude that
these measures are a fortiori  concentrated around the infimum of the classical
energy.

\noindent
\textit{Proof of Theorem \ref{main.th.1}:}\\
Our goal is to identify
the Wigner measures of the evolved state ${\varrho}_{\varepsilon}(t)=e^{-i \frac{t}{\varepsilon} H}\varrho_{\varepsilon}e^{i \frac{t}{\varepsilon} H}$ given in
Theorem \ref{main.th.1}.
However, instead of considering ${\varrho}_{\varepsilon}(t)$, we work in the interaction representation with
$$
 \tilde{\varrho}_{\varepsilon}(t)=e^{i \frac{t}{\varepsilon} H_0} {\varrho}_{\varepsilon}(t)
 e^{-i \frac{t}{\varepsilon} H_0}\,.
$$
By doing so, we require less  regularity on the state $\varrho_{\varepsilon}$ and it is still easy to recover Wigner measures of ${\varrho}_{\varepsilon}(t)$ from those of
$\tilde{\varrho}_{\varepsilon}(t)$. The main point now is that Wigner measures of the latter states are determined through all possible "limit points", when $\varepsilon\to 0$, of the map
\begin{equation}
\label{quan2}
 \xi\mapsto \Tr \Bigl[ \tilde{\varrho}_{\varepsilon}(t) W(\xi)\Bigr] \,.
\end{equation}
Despite its apparent simplicity, there is no straightforward  way to compute such limit explicitly.
Moreover, uniqueness of Wigner measures at each time $t$ is not guarantied even if it is assumed at the initial time $t=0$ (i.e.: the map \eqref{quan2} may have several limit points though it has one single limit at $t=0$). To overcome the last difficulty, we use a diagonal extraction (or Ascoli type) argument which implies that for any sequence $(\tilde{\varrho}_{\varepsilon_n})_{n\in\mathds{N}}$, $\varepsilon_n\to 0$, we can extract a subsequence $(\tilde{\varrho}_{\varepsilon_{n_k}})_{k\in\mathds{N}}$ such that for each time, $t\in\mathds{R}$,  $(\tilde{\varrho}_{\varepsilon_{n_k}}(t))_{k\in\mathds{N}}$  admits a unique Wigner measure denoted by $\tilde\mu_t$.

The next step is to observe that \eqref{quan2} satisfies a dynamical equation which when $\varepsilon\to 0$ leads to a well behaved classical dynamical equation on the inverse-Fourier transform of the Wigner measures $\tilde\mu_t$. By integrating with respect to appropriate trial functions, we obtain a natural transport (Liouville) equation satisfied by $\tilde\mu_t$. Therefore, it is possible to identify the measures $\tilde\mu_t$ if we can prove that such transport equation has a unique solution for each data $\tilde\mu_0$ given by the push-forward  of $\tilde\mu_0$ by the corresponding classical dynamics.
To sum up, the outline of the proof goes as follows:
\begin{enumerate}[1)]
  \item We justify the integral (or Duhamel) formula
  \begin{equation*}
    \Tr \Bigl[ \tilde{\varrho}_{\varepsilon}(t) W(\xi)\Bigr] = \Tr \Bigl[\varrho_{\varepsilon} W(\xi)\Bigr] + \frac{i}{\varepsilon}\int_0^t\Tr \Bigl[ \varrho_{\varepsilon}(s) [H_I,W(\tilde{\xi}(s))] \Bigr]  ds\; ,
  \end{equation*}
in Proposition \ref{prop:4} for states ${\varrho}_{\varepsilon}$ satisfying a strong  regularity condition, namely that it belongs to the space $\mathcal{T}_{\varepsilon}^{1}$ given in Definition \ref{def:3}.
  \item By explicit computation and taking care of domain problems, we show in
  Proposition \ref{prop:5} that
   \begin{equation}
   \label{eq.29}
    \Tr \Bigl[ \tilde{\varrho}_{\varepsilon}(t) W(\xi)\Bigr] = \Tr \Bigl[\varrho_{\varepsilon} W(\xi)\Bigr] + \sum_{j=0}^2\varepsilon^j\int_0^t\Tr \Bigl[ \varrho_{\varepsilon}(s)W( \tilde{\xi}(s)) B_j( \tilde{\xi}(s))  \Bigr]  ds\;;
  \end{equation}
  where $B_j( \tilde{\xi}(s))$ are operators given in~\emph{(\ref{eq:12})-(\ref{eq:14})}.
  \item There is no loss of generality if we assume that $({\varrho}_{\varepsilon})_{\varepsilon\in(0,\bar{\varepsilon})}$ has a single Wigner measure $\mu_0$. Moreover, we prove as explained before that from any sequence $\varepsilon_n\to 0$ we can extract a subsequence $(\varepsilon_{n_k})_{k\in\mathds{N}}$ such that $(\tilde{\varrho}_{\varepsilon_{n_k}}(t))_{k\in\mathds{N}}$ has a single Wigner measure $\tilde\mu_t$ for each time $t\in\mathds{R}$ (see Subsection \ref{sec:subs-extr-all}).
  \item Letting $\varepsilon_{n_k}\to 0$ in \eqref{eq.29} and using some elementary $\varepsilon$-uniform estimates proved in Section \ref{sec:dynam-quant-class} with some Wigner measures properties; we show in Proposition \ref{prop:13} that
   \begin{equation*}
      \begin{split}
        \tilde\mu_t(e^{i \sqrt{2}\Re\langle {\xi} ,\,\cdot\,
          \rangle_{}})=
        \mu_0(e^{i \sqrt{2}\Re\langle {\xi},
          \,\cdot\,\rangle_{}})+i\sqrt{2}\int_0^t \tilde\mu_{s}\left(e^{i
            \sqrt{2}\Re\langle {\xi} , z\rangle_{}}  {\Re}\langle
          \xi, \mathscr{V}_{s}(z)\rangle\right)\; ds;
      \end{split}
    \end{equation*}
    with a velocity vector field $\mathscr{V}_{s}(z)$ defined by \eqref{eq:21}.
  \item In Proposition \ref{prop:liouville}, we show that  $t\in\mathds{R}\mapsto\tilde{\mu}_{t}$ is a weakly narrowly continuous map valued on probability measures satisfying the transport equation
  \begin{equation*}
    \partial_{t}\tilde{\mu}_t+\nabla^T\left(\mathscr{V}_{t} \tilde{\mu}_{t}\right)=0\,,
  \end{equation*}
  understood in the weak sense,
  \begin{equation*}
    \int_{\mathds{R}}\int_{\Z}
    \left(\partial_{t}f+\Re\langle \nabla f, \mathscr{V}_{t}\rangle\right)~d\tilde{\mu}_{t}dt=0\,.
  \end{equation*}
  \item To identify the measures $\tilde\mu_t$ we rely on an argument worked out in finite dimension by \citet{AGS} for the purpose of optimal transport theory and extended in \cite{2011arXiv1111.5918A} to an infinite dimensional Hilbert space setting. This yields, in Proposition \ref{prop:regdata}, the result of Theorem \ref{main.th.1} but under a strong assumption on
      $(\varrho_{\varepsilon})_{\varepsilon\in(0,\bar{\varepsilon})}\in\cap_{\delta>0}
  \mathcal{T}^{\delta}\cap \mathcal{S}^1$,  given in  Definition \ref{def:2}.  
 \item To complete the proof, we use an approximation argument allowing to extend the previous result to states satisfying the weak assumption \eqref{eq:0} in Theorem \ref{main.th.1} (see Section \ref{se.gendata}).
\end{enumerate}

\section{Dynamics, quantum and classical.}
\label{sec:dynam-quant-class}

In this section we provide informations on the dynamics of the Nelson
model with cut off and its classical counterpart. Most results are
proved in detail in \cite{Fa}, in the case $d=3$; and such results
extend immediately to any dimension. We will briefly outline the
proofs here, for the reader's convenience.

\subsection{Quantum system.}
\label{sec:quantum-system}

The phase space of the theory is $\mathscr{Z}:=L^2(\mathds{R}^d)\oplus
L^2(\mathds{R}^d)$, and we construct the Fock space
$\mathscr{H}:=\Gamma_s(\mathscr{Z})\sim
\Gamma_s(L^2(\mathds{R}^d))\otimes\Gamma_s(L^2(\mathds{R}^d))$.  The
Nelson Hamiltonian $H$ as well as the annihilation-creation of the
nucleon field $\psi^{\#}$ and the meson field $a^{\#}$ are recalled in
the introduction. As mentioned in the introduction, it is useful to
set:
\begin{align*}
  H_{01}&:=d\Gamma_1(-\frac{\Delta}{2M}+V)=\frac{1}{2M}\int_{\mathds{R}^d}^{}(\nabla\psi)^{*}(x)\nabla\psi(x)dx+\int_{\mathds{R}^d}^{}V(x)\psi^{*}(x)\psi(x)dx\; ,\\
  H_{02}&:=d\Gamma_2(\omega)=\int_{\mathds{R}^d}^{}\omega(k)a^{*}(k)a(k)dk\; ,\\
  H_0&:=H_{01}+H_{02}\; ,\\
  H_I&:=\int_{\mathds{R}^{2d}}^{}\frac{\chi(k)}{\sqrt{\omega(k)}}\psi^{*}(x)\bigl(a^{*}(k)e^{-ik\cdot x}+a(k)e^{ik\cdot x}\bigr)\psi(x) dkdx\; .
\end{align*}
We remark that, with our assumption \eqref{eq:5} on $V$, $H_0$ is a
positive self-adjoint operator on $\Gamma_s(\mathscr{Z})$ with its
natural domain $D(H_0)$.

Let $N_1=d\Gamma_{1}(1)=\int_{\mathds{R}^d} \psi^{*}(x)\psi(x) dx$ and
$N_2=d\Gamma_{2}(1)=\int_{\mathds{R}^d} a^{*}(k)a(k) dk$ be the number
operators. Since $H$ commutes with $N_1$ it is natural to split the
Fock space into sectors with a fixed number of non-relativistic
particles; hence we define the subspace
\begin{equation}
  \label{eq:6}
  \mathscr{H}_{n}:= L_s^2(\mathds{R}^{n d})\otimes\Gamma_s(L^2(\mathds{R}^d)) \; .
\end{equation}
Here $L_s^2(\mathds{R}^{n d})$ denotes the space of symmetric square
integrable functions (i.e.:
$\phi(x_{1},\cdots,x_{n})=\phi(x_{\sigma_{1}},\cdots,x_{\sigma_n})$
for any permutation $\sigma$).  By definition, we have:
\begin{equation*}
  \mathscr{H}=\bigoplus_{n=0}^{\infty}\mathscr{H}_{n}\; .
\end{equation*}
Given an operator $X$ on $\mathscr{H}$, we call $X_{n}$ its
restriction to the subspace $\mathscr{H}_{n}$. The restriction on
$\mathscr{H}_{n}$ of the operator $a(f)$, $f\in L^2(\mathds{R}^d)$,
can be extended to any function $f\in
L^{\infty}(\mathds{R}^{nd},L^2(\mathds{R}^d))$; we will denote
$a(f)_{n}$ again by $a(f)$ if no confusion arises.

\begin{lemma}
  \label{lemma:2}
  \begin{enumerate}[i)]
  \item Let $f\in L^{\infty}(\mathds{R}^{nd},L^2(\mathds{R}^d))$ such
    that $\omega^{-1/2}f\in
    L^{\infty}(\mathds{R}^{nd},L^2(\mathds{R}^d))$. Then for any
    $\phi\in D(H_{02}^{1/2})\cap \mathscr{H}_{n}$:
    \begin{align}
      \label{eq:7}
      \lVert a(f)\phi \rVert_{}^2&\leq \lVert \omega^{-1/2}f \rVert_{L^{\infty}(\mathds{R}^{nd},L^2(\mathds{R}^d))}^2\lVert H_{02}^{1/2}\phi\rVert_{}^2\\
      \label{eq:8}
      \lVert a^{*}(f)\phi \rVert_{}^2&\leq \lVert \omega^{-1/2}f\rVert_{L^{\infty}(\mathds{R}^{nd},L^2(\mathds{R}^d))}^2\lVert H_{02}^{1/2}\phi\rVert_{}^2+\varepsilon\lVert f\rVert_{L^{\infty}(\mathds{R}^{nd},L^2(\mathds{R}^d))}^2\lVert \phi\rVert_{}^2\; .
    \end{align}
  \item Let $f\in L^{\infty}(\mathds{R}^{nd},L^2(\mathds{R}^d))$. Then
    for any $\phi\in D(N_2^{1/2})\cap \mathscr{H}_{n}$:
    \begin{align}
      \label{eq:9}
      \lVert a(f)\phi \rVert_{}&\leq \lVert f\rVert_{L^{\infty}(\mathds{R}^{nd},L^2(\mathds{R}^d))}\lVert N_2^{1/2}\phi\rVert_{}^{}\\
      \label{eq:10}
      \lVert a^{*}(f)\phi \rVert_{}&\leq \lVert f\rVert_{L^{\infty}(\mathds{R}^{nd},L^2(\mathds{R}^d))}\lVert (N_2+\varepsilon)^{1/2}\phi\rVert_{}^{}\; .
    \end{align}
  \end{enumerate}
\end{lemma}
The proof of this lemma is standard and follows by means of a direct
calculation on $\mathscr{H}_{n}$ (see e.g. \cite{GNV}).
\begin{corollary}
  \label{cor:1}
  Let $\chi$ such that $\omega^{-1/2}\chi\in L^2(\mathds{R}^d)$. Then
  for any $\phi\in D(N_1^2+N_2)$:
  \begin{equation*}
    \lVert H_{I}\phi\rVert_{}^{}\leq \lVert \omega^{-1/2}\chi\rVert_2^{}\lVert (N_1^2+N_2+\varepsilon)\phi\rVert_{}^{}\; .
  \end{equation*}
\end{corollary}

We have now all the ingredients to prove the essential
self-adjointness of $H$.
\begin{proposition}[self-adjointness/Kato perturbation]
  \label{prop:1}
  Assume that \eqref{eq:5} holds.  Furthermore, let $\chi$ such that
  $\omega^{-1}\chi\in L^2(\mathds{R}^d)$. Then $H$ is self-adjoint on
  $\mathscr{H}$ with domain:
  \begin{equation*}
    D(H)=\{\phi\in\mathscr{H}; \forall n\in \mathds{N}, \phi_{n}:=\phi\bigr\rvert_{\mathscr{H}_{n}}\in D(H_0)\cap \mathscr{H}_{n}\text{ and } \sum_{n=0}^{\infty}\lVert H_{n}\phi_{n}\rVert_{}^2<\infty \}\; .
  \end{equation*}
\end{proposition}
\begin{proof}
  The operator $H_{n}$ is self-adjoint on $\mathscr{H}_{n}$ with
  domain $D(H_0)\cap \mathscr{H}_{n}$ since, by Lemma~\ref{lemma:2},
  $(H_I)_{n}$ is a Kato perturbation of $(H_0)_{n}$. Furthermore, the
  small constant in the perturbation is independent of $n$, hence we
  can define the self-adjoint extension of $H$ as the direct sum
  $\bigoplus_{n=0}^{\infty} H_{n}$ \citep[see][Proposition IV.1]{Fa}.
\end{proof}
\begin{remark}
  \label{rem:3}
  It is usual to assume $\chi(k)$ to be a characteristic function
  $1_{\{\lvert k\rvert\leq \kappa\}}(k)$, for some $\kappa>0$. If $m_0
  =0$ and $\chi=1_{\{\lvert k\rvert\leq \kappa\}}$, then for all $
  d\geq 3$, $\omega^{-1/2}\chi\in L^2(\mathds{R}^d)$ and
  $\omega^{-1}\chi\in L^2(\mathds{R}^d)$; hence $H$ is
  self-adjoint. However, if $d=2$ then $\omega^{-1}\chi \notin
  L^2(\mathds{R}^2)$. With a different approach, we can relax the
  requirement on $\chi$ and prove essential self-adjointness of $H$
  under the sole assumption \eqref{eq:5}.
\end{remark}
Define $\mathscr{F}_0 \subset \Gamma_s(\mathscr{Z})$ to be subspace of
finite particle vectors of $\Gamma_s(\mathscr{Z})$ (i.e.: vectors with
finite number of nucleons and mesons).
\begin{proposition}[self-adjointness/direct proof]
  \label{prop:14}
  Assume that \eqref{eq:5} holds.  Then $H$ is essentially
  self-adjoint on $D(H_0)\cap \mathscr{F}_0$. We denote the
  self-adjointness domain of $H$ as $D(H)$.
\end{proposition}
\begin{proof}
  Let $\phi\in\Gamma_s(\mathscr{Z})$. Then we denote by
  $\phi_{n_1,n_2}$ its restriction to $L^2_s(\mathds{R}^{n_1
    d})\otimes L^2_s(\mathds{R}^{n_2 d})$. Define the orthogonal
  projector $P_{\nu_1,\nu_2}\in\mathcal{L}(\Gamma_s(\mathscr{Z}))$,
  $\nu_1,\nu_2\in\mathds{N}$ by:
  \begin{equation*}
    (P_{\nu_1,\nu_2}\phi)_{n_1,n_2}=\left\{\begin{aligned}&\phi_{\nu_1,n_2}&\text{ if $n_1=\nu_1$ and $n_2\leq \nu_2$}\\ &0&\text{ otherwise}\end{aligned}\right.\; .
  \end{equation*}
  The operator $H$ is symmetric; we will prove that
  $(\zeta-H)(D(H_0)\cap \mathscr{F}_0)$ is dense in
  $\Gamma_s(\mathscr{Z})$ for all $\zeta\in \mathds{C}$ with $\Im
  \zeta\neq 0$. Let $\zeta$ such that $\Im \zeta\neq 0$; consider
  $\eta\in \Gamma_s(\mathscr{Z})$ such that for any $\phi\in
  D(H_0)\cap \mathscr{F}_0$:
  \begin{equation}
    \label{eq:17}
    \langle \eta , (\zeta-H)\phi\rangle_{}=0\; .
  \end{equation}
  If equation \eqref{eq:17} holds only for $\eta = 0$, then
  $(\zeta-H)(D(H_0)\cap \mathscr{F}_0)$ is dense in
  $\Gamma_s(\mathscr{Z})$. Equation~\eqref{eq:17} also implies:
  \begin{equation*}
    \langle \eta , H_0\phi \rangle_{}=\zeta\langle \eta , \phi\rangle_{}-\langle \eta , H_I\phi\rangle_{}\; .
  \end{equation*}
  Let $n_1,n_2\in\mathds{N}$; we choose $\phi_{n_1,n_2}\in
  D(H_{0}\rvert_{n_1,n_2})$ as $\phi$ ($H_0\rvert_{n_1,n_2}$ is the
  restriction of $H_0$ to $L^2_s(\mathds{R}^{n_1 d})\otimes
  L^2_s(\mathds{R}^{n_2 d})$). Then
  \begin{equation*}
    \begin{split}
      \langle \eta_{n_1,n_2} , H_0\phi_{n_1,n_2} \rangle_{}=\zeta\langle \eta_{n_1,n_2} , \phi_{n_1,n_2}\rangle_{}-\varepsilon\sum_{j=1}^{n_1}\Bigl(\langle \bigl(a(\omega^{-1/2}\chi e^{-ik\cdot x_j})\eta\bigr)_{n_1,n_2} , \phi_{n_1,n_2}\rangle_{}\\
      +\langle \bigl(a^{*}(\omega^{-1/2}\chi e^{-ik\cdot x_j})\eta\bigr)_{n_1,n_2} , \phi_{n_1,n_2}\rangle_{}\Bigr)\; .
    \end{split}
  \end{equation*}
  Hence
  \begin{equation}
    \label{eq:18}
    \lvert \langle \eta_{n_1,n_2} , H_0\phi_{n_1,n_2} \rangle_{}\rvert_{}^{}\leq \Bigl(\lvert \zeta\rvert_{}^{}+\varepsilon^{3/2}n_1(n_2+1)^{1/2}\lVert \omega^{-1/2}\chi \rVert_2^{}(\lVert \eta_{n_1,n_2-1} \rVert_{}^{}+\lVert \eta_{n_1,n_2+1}
    \rVert_{}^{})\Bigr)\lVert \phi_{n_1,n_2} \rVert_{}^{}\; .
  \end{equation}
  Since $\eta\in \Gamma_s(\mathscr{Z})$, \eqref{eq:18} implies
  $\eta_{n_1,n_2}\in D(H_0\rvert_{n_1,n_2})$ for all
  $n_1,n_2\in\mathds{N}$. Then $P_{\nu_1,\nu_2}\eta\in D(H_0)\cap
  \mathscr{F}_0$, for all $\nu_1,\nu_2\in\mathds{N}$. Consider now
  equation~\eqref{eq:17}; since it holds for all $\phi\in D(H_0)\cap
  \mathscr{F}_0$, we can choose $\phi=P_{\nu_1,\nu_2}\eta$. Then:
  \begin{equation*}
    \begin{split}
      \Im \langle \eta , H_0P_{\nu_1,\nu_2}\eta\rangle_{}=(\Im \zeta) \lVert P_{\nu_1,\nu_2}\eta \rVert_{}^2-\varepsilon\Im \Bigl(\sum_{j=1}^{\nu_1}\langle \eta , \bigl(a^{*}(\omega^{-1/2}\chi e^{-ik\cdot x_j})+a(\omega^{-1/2}\chi e^{-ik\cdot x_j})\bigr)P_{\nu_1,\nu_2}\eta \rangle_{}\Bigr)\; .
    \end{split}
  \end{equation*}
  $P_{\nu_1,\nu_2}$ commutes with $H_0$, hence we obtain:
  \begin{equation*}
    (\Im \zeta) \lVert P_{\nu_1,\nu_2}\eta \rVert_{}^2=\varepsilon\sum_{j=1}^{\nu_1}\Im \langle (1-P_{\nu_1,\nu_2})\eta , \bigl(a^{*}(\omega^{-1/2}\chi e^{-ik\cdot x_j})+a(\omega^{-1/2}\chi e^{-ik\cdot x_j})\bigr)P_{\nu_1,\nu_2}\eta \rangle_{}\; .
  \end{equation*}
  Now we use the following two facts:
  $a(f)P_{\nu_1,\nu_2}=P_{\nu_1,\nu_2-1}a(f)$, and
  $P_{\nu_1,\nu_2-1}(1-P_{\nu_1,\nu_2})=0$. Then:
  \begin{equation*}
    \begin{split}
      (\Im \zeta) \lVert P_{\nu_1,\nu_2}\eta \rVert_{}^2=\varepsilon\sum_{j=1}^{\nu_1}\Im \langle P_{\nu_1,\nu_2+1}(1-P_{\nu_1,\nu_2})\eta , a^{*}(\omega^{-1/2}\chi e^{-ik\cdot x_j})P_{\nu_1,\nu_2}\eta \rangle_{}\\
      =\varepsilon\sum_{j=1}^{n_1}\Im \langle a(\omega^{-1/2}\chi e^{-ik\cdot x_j})\eta_{\nu_1,\nu_2+1} , \eta_{\nu_1,\nu_2}\rangle_{}\; .
    \end{split}
  \end{equation*}
  Taking the absolute value we obtain:
  \begin{equation*}
    \lvert \Im \zeta\rvert_{}^{}\lVert P_{\nu_1,\nu_2}\eta \rVert_{}^2\leq \varepsilon^{3/2}\nu_1(\nu_2+1)^{1/2}\lVert \omega^{-1/2}\chi \rVert_2^{}\lVert \eta_{\nu_1,\nu_2+1} \rVert_{}^{}\lVert \eta_{\nu_1,\nu_2} \rVert_{}^{}\; ,
  \end{equation*}
  hence
  \begin{equation*}
    \frac{1}{(\nu_2+1)^{1/2}}\sum_{n_2=0}^{\nu_2}\lVert \eta_{\nu_1,n_2} \rVert_{}^2\leq \varepsilon^{3/2}\frac{\nu_1}{\lvert \Im \zeta\rvert_{}^{}}\lVert \omega^{-1/2}\chi \rVert_2^{}\frac{1}{2}\bigl(\lVert \eta_{\nu_1,\nu_2+1} \rVert_{}^2+\lVert
    \eta_{\nu_1,\nu_2} \rVert_{}^2)\; .
  \end{equation*}
  We define now:
  \begin{equation*}
    S:=\sum_{n_2=0}^{\infty}\lVert \eta_{\nu_1,n_2} \rVert_{}^2=\lVert P_{\nu_1,\infty}\eta \rVert_{}^2\; ;
  \end{equation*}
  where $P_{\nu_1,\infty}$ is the orthogonal projector on
  $\mathscr{H}_{\nu_1}$. Then exists a $\bar{\nu}_2$ such that for all
  $\nu_2\geq \bar{\nu}_2$:
  \begin{equation*}
    \frac{1}{2}S\leq\sum_{n_2=0}^{\nu_2}\lVert \eta_{\nu_1,n_2} \rVert_{}^2\leq S\; .
  \end{equation*}
  So for all $\nu_2\geq \bar{\nu}_2$:
  \begin{equation*}
    \frac{1}{(\nu_2+1)^{1/2}}S\leq \varepsilon^{3/2}\frac{\nu_1}{\lvert \Im \zeta\rvert_{}^{}}\lVert \omega^{1/2}\chi \rVert_2^{}\bigl(\lVert \eta_{\nu_1,\nu_2+1} \rVert_{}^2+\lVert \eta_{\nu_1,\nu_2} \rVert_{}^2)\; ;
  \end{equation*}
  taking now the sum in $\nu_2$ it becomes:
  \begin{equation*}
    S \sum_{\nu_2=\bar{\nu}_2}^{\bar{\nu}_2'}\frac{1}{(\nu_2+1)^{1/2}}\leq 2S \varepsilon^{3/2}\frac{\nu_1}{\lvert \Im \zeta\rvert_{}^{}}\lVert \omega^{-1/2}\chi\rVert_2^{}\; ,
  \end{equation*}
  for all $\bar{\nu}_2'\geq \bar{\nu}_2$. If $S\neq 0$, we have an
  absurd, since $\sum_{\nu_2\geq \bar{\nu}_2}(\nu_2+1)^{-1/2}$ is
  divergent. It follows that $(\forall \nu_1\in\mathds{N},
  P_{\nu_1,\infty}\eta=0)\Leftrightarrow \eta=0$.
\end{proof}

Finally, we describe some properties of $H$ and the corresponding
evolution $e^{-i \frac{t}{\varepsilon} H}$ in mapping domains of
particular operators in $\mathscr{H}$.

\begin{proposition}
  \label{prop:2}
  Assume that \eqref{eq:5} holds. Then:
  \begin{enumerate}[i)]
  \item\label{item:1} $D(H_0)\cap D(N_1^2+N_2)\subseteq D(H)$.
  \item\label{item:2} $D(H)\cap D(N_1^2+N_2)\subseteq D(H_0)$.
  \item\label{item:3} Let $\delta\in\mathds{R}$, $t\in\mathds{R}$ and
    $m_{\delta}(\varepsilon):=
    \max\{2+\varepsilon,1+(1+\varepsilon)^{\delta}\}$. Then for any
    $\phi\in\mathscr{H}$:
    \begin{equation*}
      \begin{split}
        \lVert (N_1^2+N_2+\varepsilon)^{\delta}e^{-i \frac{t}{\varepsilon} H}(N_1^2+N_2+\varepsilon)^{-\delta}\phi\rVert_{}^{}\leq e^{m_{\delta}(\varepsilon)\sqrt{\varepsilon}\lvert \delta\rvert^{}\lvert t\rvert^{}\lVert \omega^{-1/2}\chi\rVert_2^{}}\lVert \phi\rVert_{}^{}\; .
      \end{split}
    \end{equation*}
  \end{enumerate}
\end{proposition}
\begin{proof}
  i) From $H=H_0+H_I$ we obtain $\lVert H\phi\rVert_{}^{}\leq \lVert
  H_0\phi\rVert_{}^{}+\lVert \omega^{-1/2}\chi\rVert_2^{}\lVert
  (N_1^2+N_2+\varepsilon)\phi \rVert_{}^{}$ by Lemma~\ref{lemma:2}.\\
  ii) From $H_0=H-H_I$ we obtain $\lVert H_0\phi\rVert_{}^{}\leq
  \lVert H\phi\rVert_{}^{}+\lVert \omega^{-1/2}\chi\rVert_2^{}\lVert
  (N_1^2+N_2+\varepsilon)\phi \rVert_{}^{}$.\\
  iii) We define, for $\delta<-1/2$, $M(t):=\lVert
  (N_1^2+N_2+\varepsilon)^{\delta}e^{-i \frac{t}{\varepsilon}
    H}\phi\rVert_{}^{}$. The result is then an application of
  Gronwall's lemma on $\mathscr{H}_{n}$, taking the derivative on a
  suitable domain. The result is then extended, by density, to all
  vectors of $\mathscr{H}$. Interpolating between $\delta=-1$ and
  $\delta=0$ we obtain the result for all $\delta\leq 0$; by duality
  we conclude the proof for all real $\delta$ \citep[see][Proposition
  IV.2]{Fa}.
\end{proof}

\subsection{Classical system.}
\label{sec:classical-system}

In this part we are concerned with the following partial differential
equation on the phase space $\mathscr{Z}=L^{2}(\mathds{R}^{d})\oplus
L^{2}(\mathds{R}^{d})$:
\begin{equation}
  \label{eq:1}
  \left\{
    \begin{aligned}
      i\partial_t z_1&=\Bigl(-\frac{1}{2M} \Delta +V\Bigr)z_1 + \Bigl(\int_{\mathds{R}^d}^{}\frac{\chi(k)}{\sqrt{\omega(k)}}\bigl(\bar{z}_2(k)e^{-ik\cdot x}+z_2(k) e^{ik\cdot x}\bigr)  dk\Bigr) z_1\\
      i\partial_tz_2&= \omega z_2 +\omega^{-1/2} \chi \int_{\mathds{R}^d}^{}e^{-ik\cdot x} \bar{z}_1(x) z_1(x)dx
    \end{aligned}\right .\; .
\end{equation}
This system describes a coupled Klein-Gordon ($m_0>0$)/Wave
($m_0=0$)-Schrödinger equation; it is the classical dynamics limit of
the Nelson model. In this form the second equation does not seem a
Klein-Gordon/Wave equation, however rewriting it for $A:=
\int_{\mathds{R}^d} \frac{\chi(k)}{\sqrt{\omega(k)}} \bigl(
\bar{z}_2(k) e^{-ik\cdot x} + z_2(k) e^{ik\cdot x} \bigr) dk$, we
obtain the more usual form: $(\square + m_0^2) A= -(2\pi)^{-d/2}
\mathcal{F}^{-1}(\chi) * \lvert z_1\rvert^2$.

In the case where the ultraviolet cut off is removed (i.e.: $\chi=1$),
we obtain a coupled system with an Yukawa interaction. This latter PDE
has attracted a lot of attention, see e.g. \cite{Ba,GiVe-sca,HW,Shi}.

\begin{proposition}
  \label{lemma:1}
  Assume \eqref{eq:5} holds; and let $\mathscr{Z}\ni z^0:=z_1^0\oplus
  z_2^0$. Then the Cauchy problem:
  \begin{equation*}
    \left\{
      \begin{aligned}
        i\partial_t z_1&=\Bigl(-\frac{1}{2M} \Delta +V\Bigr)z_1+ \Bigl(\int_{\mathds{R}^d}^{}\frac{\chi(k)}{\sqrt{\omega(k)}}\bigl(\bar{z}_2(k)e^{-ik\cdot x}+z_2(k) e^{ik\cdot x}\bigr)  dk\Bigr) z_1\\
        i\partial_tz_2&= \omega z_2 +\omega^{-1/2} \chi \int_{\mathds{R}^d}^{}e^{-ik\cdot x} \bar{z}_1(x) z_1(x)dx
      \end{aligned}\right . \quad
    \left\{
      \begin{aligned}
        z_1(s)&=z_1^0\\
        z_2(s)&=z_2^0
      \end{aligned}\right .
  \end{equation*}
  admits an unique global solution in
  $\mathscr{C}^0_{}(\mathds{R}^{},\mathscr{Z})$.
\end{proposition}
\begin{proof}
  Local existence is proved by means of a fixed point argument. This
  solution is then extended globally using the conservation of $\lVert
  z_1 \rVert_2^{}$ \citep[see][Proposition III.1]{Fa}.
\end{proof}

Define now the flow $\Phi(t,s)$ on $\mathscr{Z}$ as:
\begin{equation}
  \label{eq:2}
  \begin{split}
    \Phi(t,s)z(s):=\Bigl(
    \begin{smallmatrix}
      e^{-i(t-s)(-\frac{\Delta}{2M}+V)}&0\\
      0 &e^{-i(t-s)\omega}
    \end{smallmatrix}\Bigr)
    z(s) -i\int_s^t\Bigl(
    \begin{smallmatrix}
      e^{-i(t-\tau) (-\frac{\Delta}{2M}+V)}&0\\
      0 &e^{-i(t-\tau)\omega}
    \end{smallmatrix}\Bigr)\Bigl(
    \begin{smallmatrix}
      \Phi_1(z(\tau))\\
      \Phi_2(z(\tau))
    \end{smallmatrix}\Bigr) d\tau\; ,
  \end{split}
\end{equation}
with $z(\tau)$, $\tau\in[s,t]$, the
$\mathscr{C}^0_{}(\mathds{R},\mathscr{Z})$-solution of the Cauchy
problem of Proposition \ref{lemma:1}, and
\begin{align*}
  \Phi_1(z(t))&:=\Bigl(\int_{\mathds{R}^d}^{}\frac{\chi(k)}{\sqrt{\omega(k)}}\bigl(\bar{z}_2(t,k)e^{-ik\cdot x}+z_2(t,k) e^{ik\cdot x}\bigr)  dk\Bigr) z_1(t,x)\\
  \Phi_2(z(t))&:=\omega^{-1/2}(k) \chi(k) \int_{\mathds{R}^d}^{}e^{-ik\cdot x} \bar{z}_1(t,x) z_1(t,x)dx\; .
\end{align*}
The Klein-Gordon-Schrödinger equation is a Hamiltonian system and
therefore \eqref{eq:1} can be written in a more compact way, namely
\begin{equation}
  \label{eq:comp-ham}
  i\partial_{t} z=\partial_{\bar z} h(z)\,,
\end{equation}
with $h(z), z\in\Z,$ the classical hamiltonian given by
$h(z)=h_0(z)+h_I(z)$; with
\begin{align*}
  h_{0}(z)&=\langle z_1,(-\frac{\Delta}{2M}+V)
  z_1\rangle+\langle z_2,\omega(k) z_2\rangle\,,\\
  h_{I}(z)&=
  \int_{\mathds{R}^{2d}}^{}\frac{\chi(k)}{\sqrt{\omega(k)}} |z_1|^2(x)
  \bigl(\bar{z}_2(k) e^{-ik\cdot x}+z_2(k)e^{ik\cdot x}\bigr) dkdx\;.
\end{align*}
Define the free flow
\begin{equation*}
  \Phi_0(t,s)=\Phi_0(t-s)=\Bigl(
  \begin{smallmatrix}
    e^{-i (t-s)(-\frac{\Delta}{2M}+V)}&0\\
    0 &e^{-i (t-s) \omega}
  \end{smallmatrix}\Bigr)\;.
\end{equation*}
The classical field equation~\eqref{eq:comp-ham} can be written on the
interaction representation:
\begin{equation}
  \label{eq:21}
  \partial_t \tilde{z} =\mathscr{V}_{s}(\tilde{z})=-i\Phi_{0}(-t)  \partial_{\bar z}h_{I}(\Phi_0(t) \tilde{z})\; ;
\end{equation}
with the (velocity) vector field $\mathscr{V}_{s}$ continuous on $\mathscr{Z}$
and satisfying the estimate:
\begin{equation}
  \label{eq:22}
  \lVert \mathscr{V}_s(z) \rVert_{\mathscr{Z}}^{}\leq 2\lVert \omega^{-1/2}\chi \rVert_2^{}\lVert z_1 \rVert_2^{}\bigl(\lVert z_1 \rVert_2+\lVert z_2 \rVert_2^{}\bigr)\; .
\end{equation}
\begin{proposition}
  \label{prop:3}
  Assume \eqref{eq:5} holds. Then for all $t,s\in\mathds{R}$,
  $\Phi(t,s)$ given by \eqref{eq:2} is the well defined global
  continuous flow on $\mathscr{Z}=L^2(\mathds{R}^d)\oplus
  L^2(\mathds{R}^{d})$ of the Klein-Gordon-Schrödinger
  equation~\eqref{eq:1}.
\end{proposition}
\begin{proof}
  Let $z(s)$, $s\in \mathds{R}$, be in $\mathscr{Z}$ and $z(t)$ be the
  unique global $\mathscr{C}^0_{}(\mathds{R},\mathscr{Z})$-solution of
  the Cauchy problem of Proposition \ref{lemma:1}. Then
  $\mathscr{Z}\ni z(t)= \Phi(t,s)z(s)$.
\end{proof}

\section{Trace of states.}
\label{sec:trace-states}

First of all we recall some definitions. Let $\varrho_{\varepsilon}$
be a positive trace class operator with $\Tr[\varrho_{\varepsilon}]=1$
(the conditions it have to satisfy will be discussed later); then we
define
\begin{align*}
  \varrho_{\varepsilon}(t)&:= e^{-i \frac{t}{\varepsilon}H}\varrho_{\varepsilon}e^{+i \frac{t}{\varepsilon}H}\; ,\\
  \tilde{\varrho}_{\varepsilon}(t)&:= e^{+i \frac{t}{\varepsilon}H_0}\varrho_{\varepsilon}(t) e^{-i \frac{t}{\varepsilon}H_0}\; .
\end{align*}
We denote by $\mathcal{L}^1(\mathscr{H})$ the space of trace class
operators on $\mathscr{H}$.  Also, let $\mathscr{Z}\ni
\xi=\xi_1\oplus\xi_2$. Then we define the Weyl operator
\begin{equation*}
  W(\xi):=e^{i \frac{\psi^{*}(\xi_1)+\psi(\xi_1)}{\sqrt{2}}}\otimes e^{i \frac{a^{*}(\xi_2)+a(\xi_2)}{\sqrt{2}}}= W(\xi_1)\otimes W(\xi_2)\; .
\end{equation*}
We have used here the representation of $\Gamma_s(\mathscr{Z})$ as the
tensor product
$\Gamma_s(L^2(\mathds{R}^d))\otimes\Gamma_s(L^2(\mathds{R}^d))$; we
will use freely the more suitable representation of the two, depending
on the context. Finally, let $\mathscr{Z}\ni z=z_1\oplus z_2$, and
$\Phi_0(t)$ be the free flow on $\mathscr{Z}$, defined above. Then we have
\begin{equation*}
  \tilde{z}(s)=\Phi_0(s)z=\Bigl(
  \begin{smallmatrix}
    e^{-i s(-\frac{\Delta}{2M}+V)}&0\\
    0 &e^{-i s \omega}
  \end{smallmatrix}\Bigr) \Bigl(
  \begin{smallmatrix}
    z_1\\
    z_2
  \end{smallmatrix}\Bigr)\; .
\end{equation*}

In this section we give a rigorous derivation of the following
formula, crucial in the analysis of the limit $\varepsilon\to 0$:
\begin{equation}
  \label{eq:3}
  \Tr \Bigl[ \tilde{\varrho}_{\varepsilon}(t) W(\xi)\Bigr] = \Tr \Bigl[\varrho_{\varepsilon} W(\xi)\Bigr] + \frac{i}{\varepsilon}\int_0^t\Tr \Bigl[ \varrho_{\varepsilon}(s) [H_I,W(\tilde{\xi}(s))] \Bigr]  ds\; .
\end{equation}
Furthermore, we will give a characterization of the terms in the
commutator $[H_I,W(\tilde{\xi}(s))]$.
\begin{remark*}
  The estimates on this section are made more precise than what we
  need, for a possible derivation of a quantitative rate of
  convergence.
\end{remark*}

\subsection{Derivation of the integral formula.}
\label{sec:deriv-integr-form}

For convenience, let \fbox{$T:=N_1^2+N_2+1$} and \fbox{$S:=H_0+T$}
. Then we make the following definition:
\begin{definition}[$\mathcal{S}_{\varepsilon}^{\delta}$,
  $\mathcal{T}_{\varepsilon}^{\delta}$]
  \label{def:3}
  Let $\varrho_{\varepsilon}\in \mathcal{L}^1(\mathscr{H})$,
  $\varepsilon >0, \delta \in\mathds{R}$. Then
  \begin{gather*}
    \varrho_{\varepsilon}\in\mathcal{S}_{\varepsilon}^{\delta}\Leftrightarrow\lvert \varrho_{\varepsilon}\rvert_{\mathcal{S}_{\varepsilon}^{\delta}}^{} := \lvert\varrho_{\varepsilon} S^{\delta} \rvert_{\mathcal{L}^1(\mathscr{H})}^{}< +\infty\; ;\\
    \varrho_{\varepsilon}\in\mathcal{T}_{\varepsilon}^{\delta}\Leftrightarrow \lvert \varrho_{\varepsilon}\rvert_{\mathcal{T}_{\varepsilon}^{\delta}}^{} := \lvert\varrho_{\varepsilon} T^{\delta} \rvert_{\mathcal{L}^1(\mathscr{H})}^{}< +\infty\; .
  \end{gather*}
\end{definition}

Define now the subspace $\mathscr{Z}_1\subset \mathscr{Z}$ as:
\begin{equation}
  \label{eq:15}
  \mathscr{Z}_1:=\{\mathscr{Z}\ni z=z_1\oplus z_2: z_1\in H^2(\mathds{R}^d)\text{ and } \omega z_2\in L^2(\mathds{R}^d) \}\; .
\end{equation}

In order to prove~\eqref{eq:3} we need some preparatory results proved
in Appendix~\ref{sec:estim-fock-spac}. The Corollary~\ref{cor:2},
adapted to our spaces $\mathscr{\mathscr{Z}}$ and
$\Gamma_s(\mathscr{Z})$, becomes:
\begin{lemma}
  \label{lemma:5}
  \begin{enumerate}[i)]
  \item Let $\xi\in\mathscr{Z}_1$. Then $S^{-1}W(\xi) S\in
    \mathcal{L}(\mathscr{H})$. Furthermore, there exists $C>0$ such that
    \begin{equation*}
      \begin{split}
        \lvert S^{-1}W(\xi)S\rvert_{\mathcal{L}(\mathscr{H})}^{}\leq C\Bigl(1+\varepsilon\lVert \xi \rVert_{\mathscr{Z}_1}^{}+\varepsilon^2\lVert \xi \rVert_{\mathscr{Z}_1}^2+\varepsilon^3\lVert \xi_1 \rVert_2^3+\varepsilon^4\lVert \xi_1 \rVert_2^4 \Bigr)\; .
      \end{split}
    \end{equation*}
  \item Let $\xi\in\mathscr{Z}$. Then for any $\delta\in\mathds{R}$,
    $T^{-\delta}W(\xi)T^{\delta}\in\mathcal{L}(\mathscr{H})$. Furthermore,
    there exists $C(\delta,\lVert \xi \rVert_{\mathscr{Z}}^{})>0$ such that
    \begin{equation*}
      \lvert T^{-\delta}W(\xi)T^{\delta}\rvert_{\mathcal{L}(\mathscr{H})}^{}\leq C(\delta,\lVert \xi \rVert_{\mathscr{Z}}^{})(1+O(\varepsilon))\; .
    \end{equation*}
    If $\delta=1$, there exists $C>0$ such that
    \begin{equation*}
      \lvert T^{-1}W(\xi)T\rvert_{\mathcal{L}(\mathscr{H})}^{}\leq C\Bigl(1+\varepsilon\lVert \xi \rVert_{\mathscr{Z}}+\varepsilon^2\lVert \xi \rVert_{\mathscr{Z}}^2+\varepsilon^3\lVert \xi_1 \rVert_2^3+\varepsilon^4\lVert \xi_1 \rVert_2^4 \Bigr)\; .
    \end{equation*}
  \end{enumerate}
\end{lemma}
Next we consider the operators $T^{-\delta}e^{-i
  \frac{t}{\varepsilon}H}T^{\delta}$ and $S^{-1}e^{-i
  \frac{t}{\varepsilon}H}S$.
\begin{lemma}
  \label{lemma:6}
  Let $\delta\in\mathds{R}$. Then for all $t\in \mathds{R}$,
  $T^{-\delta}e^{-i
    \frac{t}{\varepsilon}H}T^{\delta}\in\mathcal{L}(\mathscr{H})$. Furthermore, there
  exists $C(\delta,t,\lVert \omega^{-1/2}\chi \rVert_2^{})>0$ such
  that for any $\varepsilon\in(0,\bar{\varepsilon})$:
  \begin{equation*}
    \lvert T^{-\delta}e^{-i\frac{t}{\varepsilon}H}T^{\delta}\rvert_{\mathcal{L}(\mathscr{H})}^{}\leq C(\delta,t,\lVert \omega^{-1/2}\chi \rVert_2^{})\; .
  \end{equation*}
\end{lemma}
\begin{proof}
  Let $\delta\in\mathds{N}$. We recall that, for any $a\geq0$,
  $n_1,n_2\in\mathds{N}$:
  \begin{equation*}
    (n_1^2+n_2+a)^{\delta}\leq (1+2^{\delta} \tilde{a})(n_1^2+n_2)^{\delta}+ a^{\delta}\; ,
  \end{equation*}
  where $\tilde{a}=\max\{a,a^{\delta -1}\}$. Hence, using
  Proposition~\ref{prop:2}, we obtain
  \begin{equation*}
    \begin{split}
      \lVert (N_1^2+N_2+1)^{\delta}e^{-i \frac{t}{\varepsilon}H}\phi \rVert_{}^{}\leq (1+2^{\delta})\lVert (N_1^2+N_2)^{\delta}e^{-i \frac{t}{\varepsilon}H}\phi \rVert_{}^{}+\lVert \phi \rVert_{}^{}\leq (1+2^{\delta})
      e^{m_{\delta}(\varepsilon)\sqrt{\varepsilon}\lvert \delta\rvert_{}^{}\lvert t\rvert_{}^{}\lVert \omega^{-1/2}\chi \rVert_2^{}}\\\lVert (N_1^2+N_2+\varepsilon)^{\delta}\phi \rVert_{}^{}+\lVert \phi
      \rVert_{}^{}\leq (1+2^{\delta})(1+2^{\delta} \tilde{\varepsilon})
      e^{m_{\delta}(\varepsilon)\sqrt{\varepsilon}\lvert \delta\rvert \lvert t\rvert_{}^{}\lVert \omega^{-1/2}\chi \rVert_2^{}}\lVert (N_1^2+N_2+1)^{\delta}\phi \rVert_{}^{}\\
      +\Bigl( (1+2^{\delta})\varepsilon^{\delta}e^{m_{\delta}(\varepsilon)\sqrt{\varepsilon}\lvert \delta\rvert\lvert t\rvert_{}^{}\lVert \omega^{-1/2}\chi \rVert_2^{}}+1\Bigr)\lVert \phi \rVert_{}^{}\; .
    \end{split}
  \end{equation*}
  Then for all $\delta\in\mathds{Z}$:
  \begin{equation}
  \label{eq.27}
    \lvert T^{-\delta}e^{-i\frac{t}{\varepsilon}H}T^{\delta}\rvert_{\mathcal{L}(\mathscr{H})}^{}\leq \Bigl( 1+(1+2^{\lvert \delta\rvert}) \bigl(1+2^{\lvert \delta\rvert}\max\{\varepsilon,\varepsilon^{\lvert \delta\rvert -1}\}+\varepsilon^{\lvert \delta\rvert}\bigr)
    e^{m_{\delta}(\varepsilon)\sqrt{\varepsilon}\lvert \delta\rvert\lvert t\rvert_{}^{}\lVert \omega^{-1/2}\chi \rVert_2^{}}\Bigr)\; .
  \end{equation}
  The result is extended by interpolation to all
  $\delta\in\mathds{R}$.
\end{proof}
\begin{lemma}
  \label{lemma:3}
  For all $t\in\mathds{R}$, $S^{-1}e^{-i \frac{t}{\varepsilon}H}S \in
  \mathcal{L}(\mathscr{H})$.
\end{lemma}
\begin{proof}
  Let $\phi_1,\phi_2\in\mathscr{H}$. Then using Lemma~\ref{lemma:6} we
  obtain:
  \begin{equation*}
    \begin{split}
      \lvert \langle \phi_1 , S^{-1}e^{i \frac{t}{\varepsilon}H}S\phi_2 \rangle_{}\rvert_{}^{}=\lvert \langle S(H+T)^{-1}(H+T)e^{-i \frac{t}{\varepsilon}H}S^{-1}\phi_1 , \phi_2\rangle_{}\rvert_{}^{}\leq \lvert \langle S(H+T)^{-1}e^{-i \frac{t}{\varepsilon}H}HS^{-1}\phi_1 ,
      \phi_2\rangle_{}\rvert\\+\Bigl( 1+(9+9\varepsilon)e^{(2+\varepsilon)\sqrt{\varepsilon}\lvert t\rvert_{}^{}\lVert \omega^{-1/2}\chi \rVert_2^{}}\Bigr)\lVert TS^{-1}\phi_1 \rVert_{}^{}\lVert S(H+T)^{-1}\phi_2 \rVert\\\leq \Bigl(\lVert \omega^{-1/2}\chi \rVert_2
      +1+(9+9\varepsilon)e^{(2+\varepsilon)\sqrt{\varepsilon}\lvert t\rvert_{}^{}\lVert \omega^{-1/2}\chi \rVert_2^{}}\Bigr)\lVert \phi_1 \rVert_{}^{}\lVert \phi_2 \rVert_{}^{}\; .
    \end{split}
  \end{equation*}
\end{proof}
We are now ready to prove the integral formula~\eqref{eq:3}.
\begin{proposition}
  \label{prop:4}
  Assume that \eqref{eq:5} holds; and let $\xi\in\mathscr{Z}$,
  $\tilde{\xi}(s)=\Phi_0(s)\xi$. Then for all
  $\varrho_{\varepsilon}\in\mathcal{T}^1_{\varepsilon}$:
  \begin{equation*}
    \Tr \Bigl[ \tilde{\varrho}_{\varepsilon}(t) W(\xi)\Bigr] = \Tr \Bigl[\varrho_{\varepsilon} W(\xi)\Bigr] + \frac{i}{\varepsilon}\int_0^t\Tr \Bigl[ \varrho_{\varepsilon}(s) [H_I,W(\tilde{\xi}(s))] \Bigr]  ds\; .
  \end{equation*}
\end{proposition}
\begin{proof}
  The formula is proved for $\xi$ in $\mathscr{Z}_1$; and for
  $\varrho_{\varepsilon}\in \mathcal{S}^1_{\varepsilon}$. The result
  is then extended by density ($\mathcal{S}^1_{\varepsilon}$ is dense
  in $\mathcal{T}^1_{\varepsilon}$ in the $\mathcal{L}^1(\mathscr{H})$
  topology).

  If we are able to differentiate, in $t$, $\Tr [
  \tilde{\varrho}_{\varepsilon}(t)W(\xi)]$ we are done. Consider then,
  for all $t,s\in\mathds{R}$:
  \begin{equation*}
    \begin{split}
      \Tr\Bigl[ \bigl(\tilde{\varrho}_{\varepsilon}(t)-\tilde{\varrho}_{\varepsilon}(s)\bigr)W(\xi) \Bigr]=\Tr\Bigl[ \bigl( e^{i \frac{t}{\varepsilon}H_0}e^{-i \frac{t}{\varepsilon}H}-e^{i \frac{s}{\varepsilon}H_0}e^{-i \frac{s}{\varepsilon}H} \bigr)\varrho_{\varepsilon}e^{i \frac{t}{\varepsilon}H}e^{-i \frac{t}{\varepsilon}H_0}W(\xi)
      \Bigr]\\+\Tr\Bigl[ e^{i \frac{s}{\varepsilon}H_0}e^{-i \frac{s}{\varepsilon}H}\varrho_{\varepsilon}\bigl( e^{i \frac{t}{\varepsilon}H}e^{-i \frac{t}{\varepsilon}H_0}-e^{i \frac{s}{\varepsilon}H}e^{-i \frac{s}{\varepsilon}H_0} \bigr)W(\xi)\Bigr]\\
      =\Tr\Bigl[ \varrho_{\varepsilon}e^{i \frac{t}{\varepsilon}H}e^{-i \frac{t}{\varepsilon}H_0}W(\xi) \bigl( e^{i \frac{t}{\varepsilon}H_0}e^{-i \frac{t}{\varepsilon}H}-e^{i \frac{s}{\varepsilon}H_0}e^{-i \frac{s}{\varepsilon}H} \bigr)\Bigr]\\
      +\Tr\Bigl[ e^{i \frac{s}{\varepsilon}H_0}e^{-i \frac{s}{\varepsilon}H}\varrho_{\varepsilon}\bigl( e^{i \frac{t}{\varepsilon}H}e^{-i \frac{t}{\varepsilon}H_0}-e^{i \frac{s}{\varepsilon}H}e^{-i \frac{s}{\varepsilon}H_0} \bigr)W(\xi)\Bigr]\\
      =\Tr\Bigl[ \varrho_{\varepsilon}SS^{-1}e^{i \frac{t}{\varepsilon}H}SS^{-1}e^{-i \frac{t}{\varepsilon}H_0}SS^{-1}W(\xi)SS^{-1} \bigl( e^{i \frac{t}{\varepsilon}H_0}e^{-i \frac{t}{\varepsilon}H}-e^{i \frac{s}{\varepsilon}H_0}e^{-i \frac{s}{\varepsilon}H} \bigr)\Bigr]\\
      +\Tr\Bigl[ e^{i \frac{s}{\varepsilon}H_0}e^{-i \frac{s}{\varepsilon}H}\varrho_{\varepsilon}SS^{-1}\bigl( e^{i \frac{t}{\varepsilon}H}e^{-i \frac{t}{\varepsilon}H_0}-e^{i \frac{s}{\varepsilon}H}e^{-i \frac{s}{\varepsilon}H_0} \bigr)W(\xi)\Bigr]\; .
    \end{split}
  \end{equation*}
  Every operation is justified since $\varrho_{\varepsilon}\in
  \mathcal{L}^1(\mathscr{H})$ and $e^{-i \frac{t}{\varepsilon}H},e^{-i
    \frac{t}{\varepsilon}H_0},W(\xi)\in\mathcal{L}(\mathscr{H})$. Now
  recall that $\varrho_{\varepsilon}\in \mathcal{S}^1_{\varepsilon}$
  and also $S^{-1}e^{-i \frac{t}{\varepsilon}H}S,S^{-1}e^{-i
    \frac{t}{\varepsilon}H_0}S,S^{-1}W(\xi)S\in\mathcal{L}(\mathscr{H})$
  by Lemmas~\ref{lemma:5} and~\ref{lemma:3} and since $S$ commutes
  with $H_0$. Then we can just look at the limits:
  \begin{gather*}
    \lim_{s\to t}\frac{1}{t-s}S^{-1}\bigl( e^{i \frac{t}{\varepsilon}H_0}e^{-i \frac{t}{\varepsilon}H}-e^{i \frac{s}{\varepsilon}H_0}e^{-i \frac{s}{\varepsilon}H} \bigr)=-S^{-1}e^{i \frac{t}{\varepsilon}H_0}H_Ie^{-i \frac{t}{\varepsilon}H}\\
    \lim_{s\to t}\frac{1}{t-s}S^{-1}\bigl( e^{i \frac{t}{\varepsilon}H}e^{-i \frac{t}{\varepsilon}H_0}-e^{i \frac{s}{\varepsilon}H}e^{-i \frac{s}{\varepsilon}H_0} \bigr)=S^{-1}e^{i \frac{t}{\varepsilon}H}H_Ie^{-i \frac{t}{\varepsilon}H_0}\; .
  \end{gather*}
  The convergence is intended in the strong topology, and we have used
  Stone's theorem. The result is finally obtained using the fact that
  \begin{equation*}
    e^{-i \frac{t}{\varepsilon}H_0}W(\xi)e^{i \frac{t}{\varepsilon}H_0}= W( \tilde{\xi}(t))\; .
  \end{equation*}
\end{proof}

\subsection{The commutator $[H_I , W(\tilde{\xi}(s))]$.}
\label{sec:comm-h_i-wtild}

Now, once the integral formula is proved, we want to give an explicit
form of the commutator $[H_I,W(\tilde{\xi}(s))]$, in particular with
respect to the dependence on $\varepsilon$, since we are interested in
the limit $\varepsilon\to 0$.

\begin{lemma}
  \label{lemma:8}
  For all $\delta\in\mathds{R}$ and $t\in\mathds{R}$:
  $\varrho_{\varepsilon}\in
  \mathcal{T}^{\delta}_{\varepsilon}\Leftrightarrow
  \varrho_{\varepsilon}(t)\in\mathcal{T}^{\delta}_{\varepsilon}$.
\end{lemma}
\begin{proof}
  The free Hamiltonian $H_0$ commutes with $T$, hence the result is a
  direct application of Lemma~\ref{lemma:6}.
\end{proof}
The next lemma can be proved using a general result on Wick quantized
operators \citep[see][]{AmNi1}, or with a strategy similar to the one
used in Lemma~\ref{lemma:4}.
\begin{lemma}
  \label{lemma:7}
  On $D(T)$ the following equality holds strongly
  for any $\xi\in\mathscr{Z}$:
  \begin{align*}
    B_{\varepsilon}(\xi):&=W^{*}(\xi)H_I W(\xi)\\
    &=\int_{\mathds{R}^{2d}}^{}\frac{\chi(k)}{\sqrt{\omega(k)}}\bigl(\psi^{*}(x)-i \frac{\varepsilon}{\sqrt{2}} \bar{\xi}_1(x)\bigr)\Bigl(\bigl(a^{*}(k) -i \frac{\varepsilon}{\sqrt{2}} \bar{\xi}_2(k)\bigr)e^{-ik\cdot x}\\
    &+\bigl(a(k) +i \frac{\varepsilon}{\sqrt{2}}
    \xi_2(k)\bigr)e^{ik\cdot x}\Bigr)
    \bigl(\psi(x) +i \frac{\varepsilon}{\sqrt{2}}\xi_1(x)\bigr)  dxdk\; .
  \end{align*}
\end{lemma}
\begin{corollary}
  \label{cor:3}
  For all $\varrho_{\varepsilon}\in \mathcal{T}^1_{\varepsilon}$ and
  $s\in\mathds{R}$:
  \begin{equation*}
    \Tr\Bigl[\varrho_{\varepsilon}(s) [H_I,W(\tilde{\xi}(s))]\Bigr]=\Tr\Bigl[\varrho_{\varepsilon}(s) W(\tilde{\xi}(s))\bigl(B_{\varepsilon}(\tilde{\xi}(s))-H_I\bigr)\Bigr]\; .
  \end{equation*}
\end{corollary}
Now we would like to write
\begin{equation}
  \label{eq:11}
  \frac{i}{\varepsilon}(B_{\varepsilon}(\tilde{\xi}(s))-H_I)=\sum_{j=0}^r \varepsilon^j B_j( \tilde{\xi}(s))
\end{equation}
for some $r\in \mathds{N}$. This can be easily done, with $r=2$,
obtaining:
\begin{equation}
  \label{eq:12}
  \begin{split}
    B_0(\xi)= \frac{i}{\sqrt{2}} \int_{\mathds{R}^{2d}}^{}\frac{\chi(k)}{\sqrt{\omega(k)}}\Bigl[\psi^{*}(x)\bigl(a^{*}(k)e^{-ik\cdot x}+a(k)e^{ik\cdot x} \bigr)\xi_1(x)-\psi(x)\bigl(a^{*}(k)e^{-ik\cdot x}+a(k)\\
    e^{ik\cdot x} \bigr)\bar{\xi}_1(x)+\psi^{*}(x)\psi(x)\bigl( \bar{\xi}_2e^{-ik\cdot x}-\xi_2e^{ik\cdot x} \bigr) \Bigr]  dxdk\; ;
  \end{split}
\end{equation}
\begin{equation}
  \label{eq:13}
  \begin{split}
    B_1(\xi)= \frac{1}{2} \int_{\mathds{R}^{2d}}^{}\frac{\chi(k)}{\sqrt{\omega(k)}}\Bigl[\psi^{*}(x)\xi_1(x)\bigl( \bar{\xi}_2(k)e^{-ik\cdot x}-\xi_2(k)e^{ik\cdot x} \bigr)+\psi(x) \bar{\xi}_1(x)\bigl(\xi_2(k)e^{ik\cdot x} - \bar{\xi}_2(k)\\
    e^{-ik\cdot x} \bigr)+\bigl(a^{*}(k)e^{-ik\cdot x}+a(k)e^{ik\cdot x} \bigr) \bar{\xi}_1(x)\xi_1(x) \Bigr]  dxdk\; ;
  \end{split}
\end{equation}
\begin{equation}
  \label{eq:14}
  B_2(\xi)= \frac{i}{2 \sqrt{2}} \int_{\mathds{R}^{2d}}^{}\frac{\chi(k)}{\sqrt{\omega(k)}} \bar{\xi}_1(x)\xi_1(x)\bigl( \xi_2(k)e^{ik\cdot x} - \bar{\xi}_2(k)e^{-ik\cdot x} \bigr)  dxdk\; .
\end{equation}
We can sum up these results in the following proposition:
\begin{proposition}
  \label{prop:5}
  Assume \eqref{eq:5} holds; and let $\xi\in\mathscr{Z}$. Then for all
  $\varrho_{\varepsilon}\in\mathcal{T}^1_{\varepsilon}$:
  \begin{equation}
    \label{eq:16}
    \Tr \Bigl[ \tilde{\varrho}_{\varepsilon}(t) W(\xi)\Bigr] = \Tr \Bigl[\varrho_{\varepsilon} W(\xi)\Bigr] + \sum_{j=0}^2\varepsilon^j\int_0^t\Tr \Bigl[ \varrho_{\varepsilon}(s)W( \tilde{\xi}(s)) B_j( \tilde{\xi}(s))  \Bigr]  ds\; ;
  \end{equation}
  where the $B_j( \tilde{\xi}(s))$ are given
  in~\emph{(\ref{eq:12})-(\ref{eq:14})}.
\end{proposition}

Finally, we give a bound on $B_1$ and $B_2$, and a more detailed
characterization of $B_0$ (since it is the term which will have non
zero limit when $\varepsilon\to 0$). We start with the bound on $B_1$
and $B_2$:
\begin{proposition}
  \label{prop:6}
  Assume that \eqref{eq:5} holds; and let $\xi\in\mathscr{Z}$,
  $s\in[0,t]$, $t\in\mathds{R}$. Then for any
  $\varrho_{\varepsilon}\in\mathcal{T}^1_{\varepsilon}$, there exists
  $C(s,\lVert \omega^{-1/2}\chi \rVert_2^{})>0$ such that for all
  $\varepsilon\in(0,\bar{\varepsilon})$:
  \begin{equation*}
    \begin{split}
      \biggl\lvert \sum_{j=1}^2\varepsilon^j\int_0^t\Tr \Bigl[ \varrho_{\varepsilon}(s)W( \tilde{\xi}(s)) B_j( \tilde{\xi}(s))  \Bigr]  ds\biggr\rvert_{}^{}\leq \varepsilon\Bigl(1+\varepsilon\lVert \xi \rVert_{\mathscr{Z}}+(\varepsilon+\varepsilon^2)\lVert \xi \rVert_{\mathscr{Z}}^2+\varepsilon^3\lVert \xi_1
      \rVert_2^3\\+(\varepsilon^2+\varepsilon^4)\lVert \xi_1 \rVert_2^4\Bigr)\int_0^tC(s,\lVert \omega^{-1/2}\chi \rVert_2^{})  ds\;\lvert \varrho_{\varepsilon}\rvert_{\mathcal{T}^1_{\varepsilon}}^{} \; .
    \end{split}
  \end{equation*}
\end{proposition}
\begin{proof}
  We have that:
  \begin{equation*}
    \begin{split}
      \biggl\lvert \sum_{j=1}^2\varepsilon^j\int_0^t\Tr \Bigl[ \varrho_{\varepsilon}(s)W( \tilde{\xi}(s)) B_j( \tilde{\xi}(s))  \Bigr]  ds\biggr\rvert\leq \sum_{j=1}^2\varepsilon^j\int_0^t\Bigl\lvert \Tr \Bigl[ \varrho_{\varepsilon}(s)W( \tilde{\xi}(s)) B_j( \tilde{\xi}(s))
      \Bigr]\Bigr\rvert_{}^{} ds\\\leq \sum_{j=1}^2\varepsilon^j\int_0^t\lvert T^{-1}B_j(\tilde{\xi}(s)) \rvert_{\mathcal{L}(\mathscr{H})}^{}\lvert T^{-1}W(\tilde{\xi}(s))T \rvert_{\mathcal{L}(\mathscr{H})}^{}\lvert T^{-1}e^{i \frac{s}{\varepsilon}H}T \rvert_{\mathcal{L}(\mathscr{H})}^{}
      \lvert \varrho_{\varepsilon}\rvert_{\mathcal{T}^1_{\varepsilon}}^{}ds \; .
    \end{split}
  \end{equation*}
  Now, since $\lVert \tilde{\xi}(s) \rVert_{\mathscr{Z}}^{}=\lVert \xi
  \rVert_{\mathscr{Z}}^{}$, we obtain:
  \begin{equation*}
    \begin{split}
      \sum_{j=1}^2\varepsilon^j\lvert T^{-1}B_j(\tilde{\xi}(s)) \rvert_{\mathcal{L}(\mathscr{H})}\leq 2\varepsilon\lVert \omega^{-1/2}\chi \rVert_2^{}\Bigl(2\lVert \xi_1 \rVert_2^{}(\frac{1}{2}\lVert \xi_1 \rVert_2^{}+\lVert \xi_2 \rVert_2^{})+\varepsilon\lVert \xi_1
      \rVert_2^2\lVert \xi_2 \rVert_2^{} \Bigr)\; .
    \end{split}
  \end{equation*}
  Also, using Lemmas \ref{lemma:5} and \ref{lemma:6}:
  \begin{gather*}
    \lvert T^{-1}W(\tilde{\xi}(s))T\rvert_{\mathcal{L}(\mathscr{H})}^{}\leq C\Bigl(1+\varepsilon\lVert \xi \rVert_{\mathscr{Z}}+\varepsilon^2\lVert \xi \rVert_{\mathscr{Z}}^2+\varepsilon^3\lVert \xi_1 \rVert_2^3+\varepsilon^4\lVert \xi_1 \rVert_2^4 \Bigr)\; ;\\
    \lvert T^{-1}e^{-i\frac{s}{\varepsilon}H}T^{1}\rvert_{\mathcal{L}(\mathscr{H})}^{}\leq C(\delta=1,s,\lVert \omega^{-1/2}\chi \rVert_2^{})(1+O(\varepsilon))\; .
  \end{gather*}
  Hence we conclude the proof by choosing a suitable constant
  $C(s,\lVert \omega^{-1/2}\chi \rVert_2^{})$.
\end{proof}

Now we analyse in detail $B_0$. We write it as:
\begin{equation*}
  B_0:= B_{-,-}+B_{-,+}+B_{+,-}+B_{+,+}+B_{+-,0}\; ,
\end{equation*}
with
\begin{equation*}
  B_{-,-}(\tilde{\xi}(s))=\int_{\mathds{R}^{2d}}^{}\frac{\chi(k)}{\sqrt{\omega(k)}}\psi(x)a(k)e^{ik\cdot x} \bar{\tilde{\xi}}_1(s,x)  dxdk\; ,
\end{equation*}
\begin{equation*}
  B_{-,+}(\tilde{\xi}(s))=\int_{\mathds{R}^{2d}}^{}\frac{\chi(k)}{\sqrt{\omega(k)}}\psi(x)a^{*}(k)e^{-ik\cdot x} \bar{\tilde{\xi}}_1(s,x)  dxdk\; ,
\end{equation*}
\begin{equation*}
  B_{+,+}(\tilde{\xi}(s))=\bigl(B_{-,-}(\tilde{\xi}(s))\bigr)^{*}\; ,
\end{equation*}
\begin{equation*}
  B_{+,-}(\tilde{\xi}(s))=\bigl(B_{-,+}(\tilde{\xi}(s))\bigr)^{*}\; ,
\end{equation*}
\begin{equation*}
  B_{+-,0}(\tilde{\xi}(s))=\int_{\mathds{R}^{2d}}^{}\frac{\chi(k)}{\sqrt{\omega(k)}}\psi^{*}(x)\psi(x)\Bigl(\tilde{\xi}_2(s,k)e^{ik\cdot x}- \bar{\tilde{\xi}}_2(s,k)e^{-ik\cdot x} \Bigr)  dxdk\; .
\end{equation*}
We want to interpret these operators as the Wick quantization of some
symbol on $\mathscr{Z}$. A detailed description of Wick quantization
in Fock space is given in \citep{AmNi1}.  We can write a symbol
$b(z)$, $z\in\mathscr{Z}$, corresponding to the product of $q$
creation and $p$ annihilation operators, as a sesquilinear form on
$(\mathscr{Z}^{\otimes^{q}_s}) \times
(\mathscr{Z}^{\otimes^{p}_s})$. Hence we associate with it an operator
$\tilde{b}$ from $\mathscr{Z}^{\otimes^p_s}$ into
$\mathscr{Z}^{\otimes^q_s}$. A special role is played by the symbols
with compact $\tilde{b}$ (we will call them compact symbols), since
their Wick quantization can be approximated by some Weyl or Anti-Wick
quantization with an $O(\varepsilon)$ error.

Apart from $B_{+-,0}$, all the operators written above turn out to
have finite rank (compact) symbol, as stated in the following
proposition.
\begin{proposition}
  \label{prop:7}
  Let $\{e_i\}_{i\in\mathds{N}}$ be an orthonormal basis of
  $L^2(\mathds{R}^d)$. Then the following statements are true for any
  $s\in\mathds{R}$:
  \begin{enumerate}[i)]
  \item $B_{-,-}(\tilde{\xi}(s))=b_{-,-}(z)^{Wick}$ with
    $b_{-,-}(z)=\langle \tilde{b}_{-,-}(\tilde{\xi}(s)) ,
    (z)^{\otimes_2} \rangle_{\mathscr{Z}^{\otimes_2}}$. Furthermore
    \begin{equation*}
      \langle \tilde{b}_{-,-}(\tilde{\xi}(s)),\,\cdot\,\rangle_{\mathscr{Z}^{\otimes_2}} =\sum_{i,j\in\mathds{N}}^{}\langle\omega^{-1/2}(k)\chi(k)e^{ik\cdot x}\tilde{\xi}_1(s,x),e_i\otimes e_j \rangle_{L^2(\mathds{R}^{2d})}
      \langle\left(\begin{smallmatrix}e_i\\0\end{smallmatrix}\right)\otimes \left(\begin{smallmatrix}0\\e_j\end{smallmatrix}\right), \,\cdot\,\rangle_{\mathscr{Z}^{\otimes_2}}
    \end{equation*}
    is a finite rank operator from $\mathscr{Z}^{\otimes_2}$ to
    $\mathds{C}$ (since $\mathds{C}$ is spanned by a single vector).
  \item $B_{+,+}(\tilde{\xi}(s))=b_{+,+}(z)^{Wick}$ with
    $b_{+,+}(z)=\langle (z)^{\otimes_2} ,
    \tilde{b}_{+,+}(\tilde{\xi}(s))(z)^{\otimes_0}
    \rangle_{\mathscr{Z}^{\otimes_2}}$. Furthermore
    \begin{equation*}
      \tilde{b}_{+,+}(\tilde{\xi}(s))=\sum_{i,j\in\mathds{N}}^{}\langle e_i\otimes e_j , \omega^{-1/2}(k)\chi(k)e^{-ik\cdot x}\tilde{\xi}_1(s,x) \rangle_{L^2(\mathds{R}^{2d})}
      \left(\begin{smallmatrix}e_i\\0\end{smallmatrix}\right)\otimes \left(\begin{smallmatrix}0\\e_j\end{smallmatrix}\right)
    \end{equation*}
    is a finite rank operator from $\mathds{C}$ to
    $\mathscr{Z}^{\otimes_2}$.
  \item $B_{-,+}(\tilde{\xi}(s))=b_{-,+}(z)^{Wick}$ with
    $b_{-,+}(z)=\langle z , \tilde{b}_{-,+}(\tilde{\xi}(s))z
    \rangle_{\mathscr{Z}}$. Furthermore
    \begin{equation*}
      \tilde{b}_{-,+}(\tilde{\xi}(s))=\sum_{i,j\in\mathds{N}}^{}\langle\omega^{-1/2}(k)\chi(k)e^{-ik\cdot x}\tilde{\xi}_1(s,x),e_i\otimes e_j \rangle_{L^2(\mathds{R}^{2d})}\lvert 0\oplus \bar{e}_j\rangle\langle e_i\oplus 0 \rvert
    \end{equation*}
    is a Hilbert-Schmidt operator on $\mathscr{Z}$.
  \item $B_{+,-}(\tilde{\xi}(s))=b_{+,-}(z)^{Wick}$ with
    $b_{+,-}(z)=\langle z , \tilde{b}_{+,-}(\tilde{\xi}(s))z
    \rangle_{\mathscr{Z}}$. Furthermore
    \begin{equation*}
      \tilde{b}_{+,-}(\tilde{\xi}(s))=\sum_{i,j\in\mathds{N}}^{}\langle e_i\otimes e_j , \omega^{-1/2}(k)\chi(k)e^{ik\cdot x}\tilde{\xi}_1(s,x) \rangle_{L^2(\mathds{R}^{2d})}\lvert e_i\oplus 0\rangle\langle 0\oplus \bar{e}_j\rvert
    \end{equation*}

    is a Hilbert-Schmidt operator on $\mathscr{Z}$.
  \end{enumerate}
\end{proposition}
\begin{proof}
  It is very easy to see that the Wick quantization of these symbols
  is the corresponding operator on $\mathscr{H}$ (formally we
  substitute each $z_1^{\#}$ with $\psi^{\#}$ and each $z_2^{\#}$ with
  $a^{\#}$, in normal ordering).

  Also, since the sum in $i,j$ is convergent,
  $\tilde{b}_{+,+}(\tilde{\xi}(s))$ is a vector of
  $\mathscr{Z}^{\otimes_2}$, hence a finite rank operator. Finally,
  \begin{equation*}
    \Tr_{\mathscr{Z}}\Bigl[\tilde{b}_{-,+}(\tilde{\xi}(s))^{*}\tilde{b}_{-,+}(\tilde{\xi}(s))\Bigr]\leq \lVert \xi_1 \rVert_2^{2}\lVert \omega^{-1/2}\chi \rVert_2^{2}\; .
  \end{equation*}
  For $b_{+,-}$ we obtain an analogous bound.
\end{proof}

The operator $B_{+-,0}$ can be seen as the second quantization of a
multiplication operator, hence its symbol is not compact. In order to
make it compact we need to use a regularization scheme. We define the
symbol $b_{+-,0}(z)$ as $b_{+-,0}(z)=\langle z ,
\tilde{b}_{+-,0}(\tilde{\xi}(s))z \rangle_{\mathscr{Z}}$ with
\begin{equation*}
  \tilde{b}_{+-,0}(\tilde{\xi}(s))=\left(\begin{smallmatrix}f(\tilde{\xi}_2(s))&0\\0&0\end{smallmatrix}\right)\; ,
\end{equation*}
\begin{equation*}
  f(\tilde{\xi}_2(s),x)=\int_{\mathds{R}^d}^{}\frac{\chi(k)}{\sqrt{\omega(k)}}\Bigl(\tilde{\xi}_2(s,k)e^{ik\cdot x}-\bar{\tilde{\xi}}_2(s,k)e^{-ik\cdot x}\Bigr)  dk\; .
\end{equation*}
Since for all $s\in\mathds{R}$, $\omega^{-1/2}\chi,\tilde{\xi}_2(s)\in
L^2(\mathds{R}^d)$, then $f(\tilde{\xi}_2(s))\in
L^{\infty}(\mathds{R}^d)$, and $\lim_{\lvert
  x\rvert_{}^{}\to\infty}f(\tilde{\xi}_2(s))=0$. We would like to use
the following compactness criterion \citep[see e.g.][]{DG,Sim}.
\begin{proposition}
  \label{prop:8}
  Let $f,g\in L^{\infty}(\mathds{R}^d)$ such that
  \begin{equation*}
    \lim_{\lvert x\rvert_{}^{}\to \infty}f(x)=0\; ,\; \lim_{\lvert \kappa\rvert_{}^{}\to \infty}g(\kappa)=0\; .
  \end{equation*}
  Also, let $g(i\partial_x)$ be the operator acting as:
  \begin{equation*}
    g(i\partial_x)u(x):=\frac{1}{(2\pi)^{d/2}}\int_{\mathds{R}^d}^{}e^{-i\kappa\cdot x}g(\kappa)\check{u}(\kappa)  d\kappa\; .
  \end{equation*}
  Then the operator $g(i\partial_x)f(x)$ on $L^2(\mathds{R}^d)$ is
  compact.
\end{proposition}

\begin{definition}[$g_m(i\partial_x)$]
  \label{def:1}
  Let $\{g_m\}_{m\in\mathds{N}}$ be a family of functions in
  $L^{\infty}(\mathds{R}^d)$, decaying to zero at infinity, satisfying
  the following properties:
  \begin{enumerate}[i)]
  \item For all $m\in\mathds{N}$; $0\leq g_m(x) \leq 1$, for all $x\in\mathds{R}^d$.
  \item $g_m(x)\to 1$ pointwise when $m\to\infty$.
  \item For all $a,b >0$, there exists $C(a) >0$ such that for all
    $m\in\mathds{N}\setminus \{0\}$: $\lVert
    (1+a\kappa^b)^{-1}(1-g_m(\kappa))\rVert_{\infty}^{}\leq
    C(a)m^{-b}$.
  \end{enumerate}
  Then the operators $g_m(i\partial_x)$ will compactify
  $f(\tilde{\xi}_2(s),x)$ in the sense of
  Proposition~\ref{prop:8}. Furthermore they will behave suitably in
  the limit $\varepsilon\to 0$.
\end{definition}
\begin{example}
  Let $g\in \mathcal{C}^{\infty}_0(\mathds{R}^d)$ such that $g=1$ if
  $\lvert x\rvert_{}^{}\leq 1$, $g=0$ if $\lvert x\rvert_{}^{}\geq 2$
  and $0\leq g\leq 1$ if $1\leq \lvert x\rvert_{}^{}\leq 2$. Define
  $g_m(x):=g(x/m)$. Then $\{g_m\}_{m\in\mathds{N}}$ satisfies
  Definition~\ref{def:1}.
\end{example}
Consider now $\Tr \Bigl[ \varrho_{\varepsilon}(s)W( \tilde{\xi}(s))
B_{+-,0}( \tilde{\xi}(s)) \Bigr]$, we can write it as:
\begin{equation*}
  \begin{split}
    \Tr \Bigl[ \varrho_{\varepsilon}(s)W( \tilde{\xi}(s)) B_{+-,0}( \tilde{\xi}(s)) \Bigr]=\Tr \Bigl[ \varrho_{\varepsilon}(s)W( \tilde{\xi}(s)) d\Gamma\left(\left(\begin{smallmatrix}g_m(i\partial_x)
          f(\tilde{\xi}_2(s),x)&0\\0&0\end{smallmatrix}\right)\right) \Bigr]+\Tr \Bigl[ \varrho_{\varepsilon}(s)W( \tilde{\xi}(s))\\ d\Gamma\left(\left(\begin{smallmatrix}\bigl(1 -g_m(i\partial_x)\bigr)
          f(\tilde{\xi}_2(s),x)&0\\0&0\end{smallmatrix}\right)\right) \Bigr]\; .
  \end{split}
\end{equation*}
The first term on the right hand side has a now compact symbol; and
thanks to the assumptions on $\{g_m\}_{m\in\mathds{N}}$ we can make the second small
when $m\to\infty$. A precise statement is given in the next lemma,
proved with the aid of Proposition~\ref{prop:9} of
Appendix~\ref{sec:estim-fock-spac}.
\begin{lemma}
  \label{lemma:9}
  Let $\xi\in\mathscr{Z}_1$, $s\in[0,t]$, $t\in\mathds{R}$. Then for
  any $\varrho_{\varepsilon}\in \mathcal{S}^1_{\varepsilon}$ and
  $\bar{\varepsilon}>0$, there exists $C(s,\lVert \omega^{-1/2}\chi
  \rVert_2^{})>0$ such that for all
  $\varepsilon\in(0,\bar{\varepsilon})$:
  \begin{equation*}
    \begin{split}
      \Bigl\lvert \Tr \Bigl[ \varrho_{\varepsilon}(s)W( \tilde{\xi}(s)) d\Gamma\left(\left(\begin{smallmatrix}\bigl(1 -g_m(i\partial_x)\bigr) f(\tilde{\xi}_2(s),x)&0\\0&0\end{smallmatrix}\right)\right) \Bigr]\Bigr\rvert \leq
      C(s,\lVert \omega^{-1/2}\chi \rVert_2^{})\lVert \xi_2 \rVert_2^{}\Bigl(1+\varepsilon\lVert \xi \rVert_{\mathscr{Z}_1}^{}\\
      +\varepsilon^2\lVert \xi \rVert_{\mathscr{Z}_1}^2+\varepsilon^3\lVert \xi_1 \rVert_2^3+\varepsilon^4\lVert \xi_1 \rVert_2^4 \Bigr)\frac{1}{m}\lvert \varrho_{\varepsilon}\rvert_{\mathcal{S}^1_{\varepsilon}}^{}\; .
    \end{split}
  \end{equation*}
\end{lemma}
\begin{proof}
  The proof is done splitting the trace in parts as usual:
  \begin{equation*}
    \begin{split}
      \Bigl\lvert \Tr \Bigl[ \varrho_{\varepsilon}(s)W( \tilde{\xi}(s)) d\Gamma\left(\left(\begin{smallmatrix}\bigl(1 -g_m(i\partial_x)\bigr) f(\tilde{\xi}_2(s),x)&0\\0&0\end{smallmatrix}\right)\right) \Bigr]\Bigr\rvert \leq
      \lvert S^{-1}e^{i \frac{s}{\varepsilon}}S \rvert_{\mathcal{L}(\mathscr{H})}^{}\lvert S^{-1}W(\tilde{\xi}(s))S \rvert_{\mathcal{L}(\mathscr{H})}^{}\\\lvert S^{-1}d\Gamma_1\bigl((1-g_{m}(i\partial_x))f(\tilde{\xi}_2(s),x)\bigr) \rvert_{\mathcal{L}(\mathscr{H})}^{}
      \lvert \varrho_{\varepsilon}\rvert_{\mathcal{S}^1_{\varepsilon}}^{}\; ;
    \end{split}
  \end{equation*}
  where $d\Gamma_1(f)=\int_{\mathds{R}^d} f(x)\psi^{*}(x)\psi(x)
  dx$. The first two terms of the right hand side are bounded by
  Lemmas~\ref{lemma:3} and~\ref{lemma:5} respectively; for the third
  one we use Proposition~\ref{prop:9} as follows:
  \begin{equation*}
    \begin{split}
      \lvert S^{-1}d\Gamma_1\bigl((1-g_{m}(i\partial_x))f(\tilde{\xi}_2(s))\bigr) \rvert_{\mathcal{L}(\mathscr{H})}\leq \lvert \bigl(d\Gamma_1(1-\frac{\Delta}{2M})+1\bigr)^{-1}d\Gamma_1\bigl((1-g_{m}(i\partial_x))f(\tilde{\xi}_2(s))\bigr) \rvert_{\mathcal{L}(\mathscr{H})}^{}\\
      \leq (1+\sqrt{2})\lvert \bigl(1+i\partial_x/\sqrt{2M}\bigr)^{-1}(1-g_m(i\partial_x))f(\tilde{\xi}_2(s)) \rvert_{\mathcal{L}(L^2(\mathds{R}^d))}^{}\\
      \leq(1+\sqrt{2})\lVert \bigl(1+\kappa/\sqrt{2M}\bigr)^{-1}(1-g_m(\kappa)) \rVert_{\infty}^{}\lVert f(\tilde{\xi}_2(s)) \rVert_{\infty}^{}\leq C(1+\sqrt{2})\lVert \omega^{-1/2}\chi \rVert_2^{}\lVert \xi_2
      \rVert_2^{}\frac{1}{m}\; ;
    \end{split}
  \end{equation*}
  where the last inequality follows from definition~\ref{def:1} of
  $\{g_m\}$. Defining the suitable global constant $C(s,\lVert
  \omega^{-1/2}\chi \rVert_2^{})$ we conclude the proof.
\end{proof}

\section{The classical limit $\varepsilon\to 0$.}
\label{sec:mean-field-limit}

Up to this point we have analysed the time evolved state
$\varrho_{\varepsilon}(t)$ at fixed
$\varepsilon\in(0,\bar{\varepsilon})$, now we will focus on the limit
$\varepsilon\to 0$. First we will introduce and discuss the results we
need about the convergence of states to Wigner measures; then study
the limit of the integral equation~\eqref{eq:3}.

\subsection{Wigner measures.}
\label{sec:wigner-measures}

In the classical limit, the density matrix $\varrho_{\varepsilon}$
behaves like a weak distribution, or probability measure, on the phase
space $\mathscr{Z}$. We give a brief introduction to infinite
dimensional semiclassical analysis and detailed results can be found
in \citep{AmNi1,AmNi2,AmNi3,2011arXiv1111.5918A}.  Here we present the results we
need most, adapted to our setting.

\begin{definition}[$\mathcal{S}^{\delta}$, $\mathcal{T}^{\delta}$]
  \label{def:2}
  Let $\bar{\varepsilon}>0$, $(\varrho_{\varepsilon})_{\varepsilon\in
    (0,\bar{\varepsilon})}\in \mathcal{L}^1(\mathscr{H})$ a family of
  normal states and $\delta \in\mathds{R}$. Then
  \begin{gather*}
    (\varrho_{\varepsilon})_{\varepsilon\in(0,\bar{\varepsilon})}\in\mathcal{S}^{\delta}\Leftrightarrow \exists C(\delta,\bar{\varepsilon})>0, \lvert (\varrho_{\varepsilon})_{\varepsilon\in (0,\bar{\varepsilon})}\rvert_{\mathcal{S}^{\delta}}^{} := \sup_{\varepsilon\in(0,\bar{\varepsilon})}\lvert\varrho_{\varepsilon} S^{\delta} \rvert_{\mathcal{L}^1(\mathscr{H})}^{}\leq C(\delta,\bar{\varepsilon})\; ;\\
    (\varrho_{\varepsilon})_{\varepsilon\in(0,\bar{\varepsilon})}\in\mathcal{T}^{\delta}\Leftrightarrow \exists C(\delta,\bar{\varepsilon})>0, \lvert (\varrho_{\varepsilon})_{\varepsilon\in (0,\bar{\varepsilon})}\rvert_{\mathcal{T}^{\delta}}^{} := \sup_{\varepsilon\in(0,\bar{\varepsilon})}\lvert\varrho_{\varepsilon} T^{\delta} \rvert_{\mathcal{L}^1(\mathscr{H})}^{}\leq C(\delta,\bar{\varepsilon})\; .
  \end{gather*}
\end{definition}
We remark that if
$(\varrho_{\varepsilon})_{\varepsilon\in(0,\bar{\varepsilon})}\in\mathcal{S}^{\delta}
(\text{respectively }\mathcal{T}^{\delta})$, then
$\varrho_{\varepsilon}\in
\mathcal{S}^{\delta}_{\varepsilon}(\text{respectively
}\mathcal{T}^{\delta}_{\varepsilon})$ for all
$\varepsilon\in(0,\bar{\varepsilon})$; furthermore the bound of
$\lvert \varrho_{\varepsilon}
\rvert_{\mathcal{S}_{\varepsilon}^{\delta}(\mathcal{T}_{\varepsilon}^{\delta})}$
is independent of $\varepsilon$. With this definition, we are ready to
introduce the Wigner measures; the following result holds for general
symmetric Fock spaces over a separable Hilbert space, and it is proved
in \citep[][Theorem 6.2]{AmNi1}.
\begin{proposition}
  \label{prop:10}
  Let
  $(\varrho_{\varepsilon})_{\varepsilon\in(0,\bar{\varepsilon})}\in
  \bigcup_{\delta> 0}\mathcal{T}^{\delta}$, i.e. there exists
  $\bar{\delta}>0$ such that
  $(\varrho_{\varepsilon})_{\varepsilon\in(0,\bar{\varepsilon})} \in
  \mathcal{T}^{\bar{\delta}}$. Then for every sequence
  $(\varepsilon_n)_{n\in\mathds{N}}\in (0,\bar{\varepsilon})$ with
  $lim_{n\to\infty}\varepsilon_n=0$, there exists a subsequence
  $(\varepsilon_{n_k})_{k\in\mathds{N}}$ and a Borel probability
  measure $\mu$ on $\mathscr{Z}$ associated with
  $(\varrho_{\varepsilon_{n_k}})_{k\in\mathds{N}}$ characterized by:
  \begin{equation*}
    \lim_{k\to\infty}\Tr\Bigl[\varrho_{\varepsilon_{n_k}}W(\xi)\Bigr]=\int_{\mathscr{Z}}^{}e^{i \sqrt{2}\Re\langle \xi , z\rangle_{}}  d\mu(z)\; ,\; \forall\xi\in\mathscr{Z}\; .
  \end{equation*}
  Furthermore $\mu$ satisfies the following property:
  \begin{equation}
    \label{eq:Test}
    \int_{\mathscr{Z}}^{}\bigl(\lVert z_1\rVert^2_{2}+\lVert  z_2\rVert_{2}^{}+1\bigr)^{2\bar{\delta}}  d\mu(z)\leq \lvert (\varrho_{\varepsilon})_{\varepsilon\in (0,\bar{\varepsilon})}\rvert_{\mathcal{T}^{\bar\delta}}^{}<+\infty\; .
  \end{equation}
\end{proposition}
\begin{definition}
  \label{def:4}
  The set of Wigner measures associated with
  $(\varrho_{\varepsilon})_{\varepsilon\in(0,\bar{\varepsilon})}\in\bigcup_{\delta>0}\mathcal{T}^{\delta}$
  is denoted by
  \begin{equation*}
    \mathscr{M}(\varrho_{\varepsilon},\varepsilon\in(0,\bar{\varepsilon}))\; .
  \end{equation*}
\end{definition}
In general, $\mathscr{M} ( \varrho_{\varepsilon} ,
\varepsilon\in(0,\bar{\varepsilon}) )$ is not constituted by a single
element; however for each countable sequence $\varepsilon_n\to 0$ we
can extract a subsequence $(\varepsilon_{n_k})_k$ such that
$\mathscr{M} ( \varrho_{\varepsilon_{n_k}} ,
(\varepsilon_{n_k})_{k\in\mathds{N}} ) = \{\mu\}$; hence we can
suppose without loss of generality that $\mathscr{M} (
\varrho_{\varepsilon} , \varepsilon\in(0,\bar{\varepsilon}) ) =
\{\mu\}$.
\begin{remark}
  \label{rem:1}
  Let
  $(\varrho_{\varepsilon})_{\varepsilon\in(0,\bar{\varepsilon})}\in
  \bigcup_{\delta> 0}\mathcal{T}^{\delta}$; with associated measure
  $\mu$. Then, using Lemma~\ref{lemma:5} and Weyl's relation, for any
  $\xi\in\mathscr{Z}$,
  $(\varrho_{\varepsilon}W(\xi))_{\varepsilon\in(0,\bar{\varepsilon})}$
  has an associated (complex) measure $\mu_{\xi}$ with
  \begin{equation*}
    d\mu_{\xi}(z)=e^{i \sqrt{2}\Re\langle \xi , z\rangle_{}}d\mu(z)\; .
  \end{equation*}
  We refer the reader to \citep{AmNi1} for further informations on
  Wigner measures of general trace class operators.
\end{remark}

The convergence of $\rho_{\varepsilon}$ holds with a large class of
operators (under suitable conditions); in particular with Wick
quantized polynomials with compact symbol. The precise statement is
the following: Let $\mathcal{P}_{p,q}^{\infty}(\mathscr{Z})$ be the
compact polynomial symbols of degree $p$ in $z$ and $q$ in $\bar{z}$;
define $\mathcal{P}^{\infty}_{alg} (\mathscr{Z})=
\bigoplus_{p,q\in\mathds{N}}^{alg} \mathcal{P}_{p,q}^{\infty}
(\mathscr{Z})$. Then the following proposition holds:
\begin{proposition}
  \label{prop:11}
  Let $(\varrho_{\varepsilon})_{\varepsilon\in(0,\bar{\varepsilon})}
  \in \bigcap_{\delta\geq 0} \mathcal{T}^{\delta}$ such that
  $\mathscr{M}(\varrho_{\varepsilon},\varepsilon\in(0,\bar{\varepsilon}))=\{\mu\}$. Then
  for any $b\in \mathcal{P}^{\infty}_{alg} (\mathscr{Z})$:
  \begin{equation*}
    \lim_{\varepsilon\to 0}\Tr\Bigl[\varrho_{\varepsilon} \, b^{Wick}\Bigr]=\int_{\mathscr{Z}}^{}\,b(z)d\mu(z)\; .
  \end{equation*}
\end{proposition}
\begin{remark}
  \label{rem:2}
  Since we have only operators bounded by $T$, we can relax the
  hypothesis of Proposition~\ref{prop:11} to states
  $(\varrho_{\varepsilon})_{\varepsilon\in(0,\bar{\varepsilon})}\in
  \mathcal{T}^1$. In this case, the result is true for compact polynomial symbols $b\in \mathcal{P}^{\infty}_{alg} (\mathscr{Z})$ such that $T^{-1/2} b^{Wick} T^{-1/2}$ is bounded uniformly in $\varepsilon\in[0,\bar{\varepsilon}]$.
\end{remark}

\subsection{Subsequence extraction for all times.}
\label{sec:subs-extr-all}

We would like to apply proposition~\ref{prop:11} to the integral
formula~\eqref{eq:16} and obtain an integral equation for the measure
$\mu_t$ associated with $\varrho_{\varepsilon}(t)$. In order to do
that we need to be able to extract the same converging subsequence at
any time $t\in\mathds{R}$. This is what we prove in the next
proposition; preceded by a preparatory lemma.

\begin{lemma}
  \label{lemma:10}
  Let $(\varrho_{\varepsilon})_{\varepsilon\in(0,\bar{\varepsilon})}
  \in \mathcal{T}^1$. Then $\tilde{G}_{\varepsilon}(t,\xi) :=
  \Tr\Bigl[\tilde{\varrho}_{\varepsilon}(t)W(\xi) \Bigr]$ is uniformly
  equicontinuous with respect to $\varepsilon\in(0,\bar{\varepsilon})$
  on bounded subsets of $\mathds{R}\times \mathscr{Z}$.
\end{lemma}
\begin{proof}
  Let $\varepsilon\in(0,\bar{\varepsilon})$. We split $\lvert
  \tilde{G}_{\varepsilon}(t,\xi)-\tilde{G}_{\varepsilon}(s,\eta)
  \rvert \leq X_1+X_2$, with
  \begin{equation*}
    X_1:= \lvert \tilde{G}_{\varepsilon}(t,\xi)-\tilde{G}_{\varepsilon}(s,\xi) \rvert\; ,\;  X_2:= \lvert \tilde{G}_{\varepsilon}(s,\xi)-\tilde{G}_{\varepsilon}(s,\eta) \rvert\; .
  \end{equation*}
  Using Proposition~\ref{prop:5}, Lemma \ref{lemma:6}, \eqref{eq.27} and the fact that
  $B_j(\tilde{\xi}(\tau))$ is bounded uniformly in $\tau$ and
  $\varepsilon\in(0,\bar{\varepsilon})$ for $j=0,1,2$, we obtain for
  some $C_1(\bar{\varepsilon},\lVert \xi
  \rVert_{\mathscr{Z}}^{}),C_2(\bar{\varepsilon},\lVert \xi
  \rVert_{\mathscr{Z}}^{})>0$:
  \begin{equation*}
    \begin{split}
      X_1=\Bigl\lvert \sum_{j=0}^2\varepsilon^j\int_s^t\Tr \Bigl[ \varrho_{\varepsilon}(\tau)W( \tilde{\xi}(\tau)) B_j( \tilde{\xi}(\tau))  \Bigr]  ds\Bigr\rvert\leq C_1\lvert e^{C_2 \lvert t \rvert}-e^{C_2\lvert s\rvert} \rvert\; .
    \end{split}
  \end{equation*}
  Consider now $X_2$; using Weyl's relation and the fact that
  $(\varrho_{\varepsilon})_{\varepsilon\in(0,\bar{\varepsilon})} \in
  \mathcal{T}^1$ we obtain, for some $C_3(s,\bar{\varepsilon})>0$:
  \begin{equation*}
    \begin{split}
      X_2\leq \lvert (W(\eta)W^{*}(\xi)-1)T^{-1}\rvert_{\mathcal{L}(\mathscr{H})}^{}\lvert\tilde{\varrho}_{\varepsilon}(s)\rvert_{\mathcal{T}^1}^{}\leq C_3\lvert (e^{i \frac{\varepsilon}{2}\Im\langle \eta , \xi\rangle }W(\eta-\xi)-1)T^{-1}\rvert_{\mathcal{L}(\mathscr{H})}^{}\\
      \leq C_3\Bigl(\lvert e^{i \frac{\varepsilon}{2}\Im \langle \eta , \xi\rangle }-1 \rvert +\lvert (W(\eta-\xi)-1)T^{-1}\rvert_{\mathcal{L}(\mathscr{H})}\Bigr)\; .
    \end{split}
  \end{equation*}
  Now, we use the following bound for the first term:
  \begin{equation*}
    \lvert e^{i \frac{\varepsilon}{2}\Im\langle  \eta, \xi\rangle}-1\rvert_{}^{}=\lvert e^{i \frac{\varepsilon}{2}\Im\langle  \eta-\xi, \xi\rangle}-1\rvert_{}^{}\leq 2\bar{\varepsilon}\lVert \xi \rVert^{}e^{\frac{\bar{\varepsilon}}{2}\lVert \xi \rVert(\lVert \eta \rVert+\lVert \xi \rVert)}\lVert \eta-\xi \rVert^{};
  \end{equation*}
  and for the second:
  \begin{equation*}
    \lvert (W(\eta-\xi)-1)T^{-1}\rvert_{\mathcal{L}(\mathscr{H})}\leq \Bigl\lvert \int_0^1W(\lambda (\eta-\xi))\varphi(\eta-\xi)T^{-1}  d\lambda\Bigr\rvert_{\mathcal{L}(\mathscr{H})}^{}\leq \sqrt{2}\lVert \xi-\eta \rVert_{}^{}\; ;
  \end{equation*}
  where
  $\sqrt{2}\varphi(z)=\bigl(\psi^{*}(z_1)+\psi(z_1)+a^{*}(z_2)+a(z_2)\bigr)$.
  Finally we obtain
  \begin{equation*}
    X_2\leq C_3\Bigl( 2\bar{\varepsilon}\lVert \xi \rVert^{}e^{\frac{\bar{\varepsilon}}{2}\lVert \xi \rVert(\lVert \eta \rVert+\lVert \xi \rVert)}+\sqrt{2}\Bigr) \lVert \xi-\eta \rVert\; .
  \end{equation*}
  Now, choose a bounded subset $I=[-T_0,T_0]\times \{z, \lVert z
  \rVert\leq R\}$, $T_0,R>0$. Then there exist $C_1,C_2,C_3>0$ that depend
  only on $T_0$, $R$ and $\bar{\varepsilon}$ such that for all
  $(t,\xi),(s,\eta)\in I$:
  \begin{equation}
    \label{eq:23}
    \lvert \tilde{G}_{\varepsilon}(t,\xi)-\tilde{G}_{\varepsilon}(s,\eta)\rvert_{}^{}\leq C_1\bigl\lvert e^{C_2\lvert t\rvert_{}^{}}-e^{C_2\lvert s\rvert_{}^{}}\bigr\rvert+C_3\lVert \xi-\eta \rVert_{}^{}\; .
  \end{equation}
\end{proof}
\begin{proposition}
  \label{prop:12}
  Let $(\varrho_{\varepsilon})_{\varepsilon\in(0,\bar{\varepsilon})}
  \in \mathcal{T}^\delta$, $\delta\geq 1$. Then for any sequence
  $(\varepsilon_n)_{n\in\mathds{N}}\in (0,\bar{\varepsilon})$,
  converging to zero, there exists a subsequence
  $(\varepsilon_{n_k})_{k\in\mathds{N}}$ and a family of Borel
  measures $(\tilde{\mu}_t)_{t\in\mathds{R}}$ on $\mathscr{Z}$ such
  that for all $t\in\mathds{R}$:
  \begin{equation*}
    \mathscr{M}(\tilde{\varrho}_{\varepsilon_{n_k}}(t),k\in\mathds{N})=\{\tilde{\mu}_t\}\; .
  \end{equation*}
  Furthermore for any $T_0\geq 0$ there exists $C(T_0)>0$ such that
  for all $t\in [-T_0,T_0]$:
  \begin{equation}
    \label{eq:24}
    \int_{\mathscr{Z}}^{}\bigl(\lVert z_1\rVert^2_{2}+\lVert  z_2\rVert_{2}^{}+1\bigr)^{2\delta}  d\tilde{\mu}_t(z)<C(T_0)\; .
  \end{equation}
\end{proposition}
\begin{proof}
  Recall that for any
  $(\varrho_{\varepsilon})_{\varepsilon\in(0,\bar{\varepsilon})} \in
  \mathcal{T}^\delta\subset \mathcal{T}^1$,
  $(\tilde{\varrho}_{\varepsilon}(t))_{\varepsilon\in(0,\bar{\varepsilon})}
  \in \mathcal{T}^1$ for all $t\in\mathds{R}$ (using
  Lemma~\ref{lemma:6} and the fact that $H_0$ commutes with $T$). The
  field $\mathds{R}$ is separable, so we can consider a dense
  countable set $D\subset \mathds{R}$. Let
  $(\varepsilon_n)_{n\in\mathds{N}}\in(0,\bar{\varepsilon})$.  We can
  choose, by a diagonal extraction argument, a \emph{single}
  subsequence $(\varepsilon_{n_k})_{k\in\mathds{N}}$ such that we can
  apply Proposition~\ref{prop:10} and obtain, for any $t_j\in D$:
  \begin{equation*}
    \lim_{k\to\infty}\Tr\Bigl[\tilde{\varrho}_{\varepsilon_{n_k}}(t_j)W(\xi)\Bigr]=\int_{\mathscr{Z}}^{}e^{i \sqrt{2}\Re\langle \xi , z\rangle_{}}  d\tilde{\mu}_{t_j}(z)=: \tilde{G}_0(t_j,\xi)\; .
  \end{equation*}
  Also, since
  $0\leq\Tr\Bigl[\tilde{\varrho}_{\varepsilon}(t_j)W(\xi)\Bigr]\leq 1$
  holds for any $\varepsilon\in(0,\bar{\varepsilon})$, then $0\leq
  \tilde{G}_0(t_j,\xi)\leq 1$. Now we can use Lemma~\ref{lemma:10} and
  obtain for all $t_j,t_l\in D$:
  \begin{equation*}
    \lvert \tilde{G}_{\varepsilon_{n_k}}(t_j,\xi)-\tilde{G}_{\varepsilon_{n_k}}(t_l,\xi) \rvert_{}^{}\leq C_1\bigl\lvert e^{C_2\lvert t_j\rvert_{}^{}}-e^{C_2\lvert t_l\rvert_{}^{}}\bigr\rvert\; ,
  \end{equation*}
  uniformly in $\varepsilon_{n_k}$, then we can take the limit
  $k\to\infty$ and obtain
  \begin{equation*}
    \lvert \tilde{G}_{0}(t_j,\xi)-\tilde{G}_{0}(t_l,\xi) \rvert_{}^{}\leq C_1\bigl\lvert e^{C_2\lvert t_j\rvert_{}^{}}-e^{C_2\lvert t_l\rvert_{}^{}}\bigr\rvert\; .
  \end{equation*}
 Let   $t\in\mathds{R}$; choose $(t_i)_{i\in\mathds{N}}\in D,$ such that
  $t_i\to t$, when $i\to\infty$. Then $(\tilde{G}_0(t_i,\xi))_{i\in\mathds{N}}$ is a Cauchy sequence and we can define
  \begin{equation*}
    \tilde{G}_0(t,\xi):=\lim_{i\to\infty}\tilde{G}_0(t_i,\xi)\; .
  \end{equation*}
  For all $t\in\mathds{R}$, $\tilde{G}_0(t,\,\cdot\,)$ is a norm continuous
  normalised function of positive type which satisfies
  \begin{equation}
    \label{eq:25}
    \lvert \tilde{G}_{0}(t,\xi)-\tilde{G}_{0}(s,\eta)\rvert_{}^{}\leq C_1\bigl\lvert e^{C_2\lvert t\rvert_{}^{}}-e^{C_2\lvert s\rvert_{}^{}}\bigr\rvert+C_3\lVert \xi-\eta \rVert_{}^{}\; ,
  \end{equation}
  on any bounded subset of $\mathds{R}\times \mathscr{Z}$, for some
  positive constants $C_1,C_2$ and $C_3$.

  Hence it is the characteristic function of a weak distribution
  $\tilde{\mu}_t$ on $\mathscr{Z}$, and for all $t\in\mathds{R}$:
  \begin{equation*}
    \lim_{k\to\infty}\Tr\Bigl[\tilde{\varrho}_{\varepsilon_{n_k}}(t)W(\xi)\Bigr]=\int_{\mathscr{Z}}^{}e^{i \sqrt{2}\Re\langle \xi , z\rangle_{}}  d\tilde{\mu}_{t}(z)\; .
  \end{equation*}
  Furthermore $\tilde{\mu}_t$ are Borel probability measures since
  they are Wigner measures of
  $(\tilde{\varrho}_{\varepsilon_{n_k}}(t))_{k\in\mathds{N}}\in
  \mathcal{T}^\delta$. The bound~\eqref{eq:24} comes from~\eqref{eq:Test}
  and Lemma~\ref{lemma:6}.
\end{proof}

\begin{corollary}
  \label{cor:4}
  The following statements are true:
  \begin{enumerate}[i)]
  \item\label{item:4} Let
    $(\varrho_{\varepsilon})_{\varepsilon\in(0,\bar{\varepsilon})} \in
    \mathcal{T}^1$. Then for any sequence
    $(\varepsilon_n)_{n\in\mathds{N}}\in (0,\bar{\varepsilon})$
    converging to zero, there exists a subsequence
    $(\varepsilon_{n_k})_{k\in\mathds{N}}$ and a family of Borel
    measures $(\mu_t)_{t\in\mathds{R}}$ on $\mathscr{Z}$ such that for
    all $t\in\mathds{R}$:
    \begin{equation*}
      \mathscr{M}(\varrho_{\varepsilon_{n_k}}(t),k\in\mathds{N})=\{\mu_t\}\; .
    \end{equation*}
  \item\label{item:5} Let
    $(\varrho_{\varepsilon})_{\varepsilon\in(0,\bar{\varepsilon})} \in
    \mathcal{T}^1$, $\xi\in\mathscr{Z}$. Then for any sequence
    $(\varepsilon_n)_{n\in\mathds{N}}\in (0,\bar{\varepsilon})$
    converging to zero, there exists a subsequence
    $(\varepsilon_{n_k})_{k\in\mathds{N}}$ and a family of Borel
    measures $(\mu_{t,\xi})_{t\in\mathds{R}}$ on $\mathscr{Z}$ such
    that for all $t\in\mathds{R}$:
    \begin{equation*}
      \mathscr{M}(\varrho_{\varepsilon_{n_k}}(t)W(\tilde{\xi}(t)),k\in\mathds{N})=\{\mu_{t,\xi}\}\; ;
    \end{equation*}
    furthermore:
    \begin{equation*}
      d\mu_{t,\xi}(z)=e^{i \sqrt{2}\Re\langle \tilde{\xi}(t) , z\rangle_{}}d\mu_t(z)\; .
    \end{equation*}
  \end{enumerate}
\end{corollary}
\begin{proof}
  \ref{item:4}) follows easily since for any
  $\varrho\in\mathcal{L}^1(\mathscr{H})$ and $\xi\in \mathscr{Z}$:
  $\Tr\Bigl[\tilde{\varrho}(t)W(\xi)\Bigr]=\Tr\Bigl[\varrho(t)W(\tilde{\xi}(t))\Bigr]$.

  \ref{item:5}) is a consequence of Remark~\ref{rem:1}.
\end{proof}

\subsection{Integral formula in the limit $\varepsilon \to 0$.}
\label{sec:integr-form-limit}

We have all the ingredients to calculate the limit $\varepsilon\to 0$
of the integral equation~\eqref{eq:16}.
\begin{proposition}
  \label{prop:13}
  Assume that \eqref{eq:5} holds and let $\xi\in\mathscr{Z}$,
  $(\varrho_{\varepsilon})_{\varepsilon\in(0,\bar{\varepsilon})}\in\mathcal{S}^1$. Then
  for any sequence $ (\varepsilon_n)_{n\in\mathds{N}}\in
  (0,\bar{\varepsilon})$ such that $\lim_{n\to\infty}
  \varepsilon_n=0$; there exists a subsequence $
  (\varepsilon_{n_k})_{k\in\mathds{N}}$, and a family
  $(\mu_t)_{t\in\mathds{R}}$ of Borel probability measures on
  $\mathscr{Z}$ such that for all $ t\in\mathds{R}$:
  \begin{enumerate}
  \item
    $\mathscr{M}(\varrho_{\varepsilon_{n_k}}(t),k\in\mathds{N})=\{\mu_t\}$
    and
    $\mathscr{M}(\tilde\varrho_{\varepsilon_{n_k}}(t),k\in\mathds{N})=\{\tilde\mu_t=\Phi_{0}(-t)_{\#}\mu_{t})\}$.
  \item $\tilde\mu_t$ satisfies the following integral equation:
    \begin{equation}
      \label{eq:transf}
      \begin{split}
        \tilde\mu_t(e^{i \sqrt{2}\Re\langle {\xi} ,\,\cdot\,
          \rangle_{}})=
        \mu_0(e^{i \sqrt{2}\Re\langle {\xi},
          \,\cdot\,\rangle_{}})+i\sqrt{2}\int_0^t \tilde\mu_{s}\left(e^{i
            \sqrt{2}\Re\langle {\xi} , z\rangle_{}}  {\Re}\langle
          \xi, \mathscr{V}_{s}(z)\rangle\right)\; ds;
      \end{split}
    \end{equation}
    with the (velocity) vector field
    $\mathscr{V}_{s}(z)=-i\Phi_{0}(-t) \partial_{\bar
      z}h_{I}(\Phi_0(t)z)$ for all $z\in\Z$.
  \end{enumerate}
\end{proposition}
\begin{proof}
  The first point is just a restatement of Corollary~\ref{cor:4}. The
  second is proved starting from the integral equation~\eqref{eq:16}
  and assuming $\xi\in\mathscr{Z}_1$. Fix the subsequence
  $(\varepsilon_{n_{k}})_{k\in\mathds{N}}$ such that we can associate
  a measure $\mu_t$ to
  $(\varrho_{\varepsilon_{n_k}}(t))_{k\in\mathds{N}}$ for all
  times. Then in~\eqref{eq:16} the left hand side and the first term
  in the right hand side converge by virtue of
  Proposition~\ref{prop:12}, and its Corollary~\ref{cor:4}. The terms
  involving $B_1$ and $B_2$ converge to zero in absolute value by
  Proposition~\ref{prop:6}, since
  $(\varrho_{\varepsilon_{n_k}})_{k\in\mathds{N}}\in
  \mathcal{T}^1\supset \mathcal{S}^1$. It remains to consider the
  $B_0$ term. If we split it as described in
  Section~\ref{sec:comm-h_i-wtild}, we see that the $B_{-,-}$,
  $B_{-,+}$, $B_{+,-}$ and $B_{+,+}$ terms converge by means of
  Proposition~\ref{prop:11} (applied to the state
  $\varrho_{\varepsilon_{n_k}}(s)W(\tilde{\xi}(s))$), since they have
  compact symbols.

  We have to be more careful with the $B_{+-,0}$ term, and use the
  regularization scheme introduced in
  Definition~\ref{def:1}. Consider:
  \begin{equation*}
    \begin{split}
      \Bigl\lvert \Tr \Bigl[ \varrho_{\varepsilon}(s)W( \tilde{\xi}(s)) B_{+-,0}( \tilde{\xi}(s))  \Bigr]-\int_{\mathscr{Z}}^{}e^{i \sqrt{2}\langle \tilde{\xi}(s) , z\rangle_{\mathscr{Z}}}\langle z_1 , f(\tilde{\xi}_2(s))z_1\rangle_{L^2(\mathds{R}^d)}  d\mu_s(z) \Bigr\rvert_{}^{}\; .
    \end{split}
  \end{equation*}
  Define now
  $B_{+-,0}^{m}(\tilde{\xi}(s)):=d\Gamma_{1}\bigl(g_m(i\partial_x)
  f(\tilde{\xi}_2(s),x)\bigr)$ to be the regularized operator with
  compact symbol. Then we obtain:
  \begin{equation*}
    \begin{split}
      \Bigl\lvert \Tr \Bigl[ \varrho_{\varepsilon}(s)W( \tilde{\xi}(s)) B_{+-,0}( \tilde{\xi}(s))  \Bigr]-\int_{\mathscr{Z}}^{}e^{i \sqrt{2}\langle \tilde{\xi}(s) , z\rangle_{\mathscr{Z}}}\langle z_1 , f(\tilde{\xi}_2(s))z_1\rangle_2  d\mu_s(z) \Bigr\rvert_{}^{}\leq \Bigl\lvert \Tr \Bigl[
      \varrho_{\varepsilon}(s)W( \tilde{\xi}(s)) \\
      B^m_{+-,0}( \tilde{\xi}(s))  \Bigr]-\int_{\mathscr{Z}}^{}e^{i \sqrt{2}\langle \tilde{\xi}(s) , z\rangle_{\mathscr{Z}}}\langle z_1 , g_m(i\partial_x)f(\tilde{\xi}_2(s))z_1\rangle_2  d\mu_s(z) \Bigr\rvert\\
      +\Bigl\lvert \Tr \Bigl[ \varrho_{\varepsilon}(s)W( \tilde{\xi}(s)) d\Gamma_1\bigl(\bigl(1 -g_m(i\partial_x)\bigr) f(\tilde{\xi}_2(s),x)\bigr) \Bigr]\Bigr\rvert\\
      +\Bigl\lvert \int_{\mathscr{Z}}^{}e^{i \sqrt{2}\langle \tilde{\xi}(s) , z\rangle_{\mathscr{Z}}}\langle z_1 , (1-g_m(i\partial_x))f(\tilde{\xi}_2(s))z_1\rangle_2  d\mu_s(z)\Bigr\rvert_{}^{}\; .
    \end{split}
  \end{equation*}
  The first term on the right hand side goes to zero by virtue of
  Proposition~\ref{prop:11}; the second goes to zero when $m\to
  \infty$ by Lemma~\ref{lemma:9}.

  Finally consider the last term. By definition, $\lvert
  (1-g_m(i\partial_x))\rvert_{\mathcal{L}(L^2(\mathds{R}^d))}^{}\leq
  1$ uniformly in $m$. Furthermore, $f(\tilde{\xi}_2(s),\,\cdot\,)\in
  L^{\infty}(\mathds{R}^d)$. Hence
  \begin{equation*}
    \bigl\lvert e^{i \sqrt{2}\langle \tilde{\xi}(s) , z\rangle_{\mathscr{Z}}}\langle z_1 , (1-g_m(i\partial_x))f(\tilde{\xi}_2(s))z_1\rangle_2\bigr\rvert_{}^{}\leq \lVert z_1 \rVert_2^2\; ,
  \end{equation*}
  that is integrable with respect to $\mu_s$ by virtue of
  Proposition~\ref{prop:10}. Then we can apply dominated convergence
  theorem and prove that the term goes to zero when $m\to\infty$,
  since $(1-g_m(i\partial_x))\to 0$ strongly as an
  operator of $L^2(\mathds{R}^d)$. \\
  Once the integral formula \eqref{eq:transf} is proved for $\xi\in
  \Z_{1}$, the extension for all $\xi\in \Z$ is straightforward since
  $\mathscr{V}_{s}$ satisfies the estimate~\eqref{eq:22} and a
  dominated convergence theorem applies thanks to the estimate
  \eqref{eq:24}.
\end{proof}

\subsection{Transport equation and uniqueness}
\label{sec:uniquness}
Proposition \ref{prop:13} shows that Wigner measures $\tilde\mu_t$ of
propagated normal states $\tilde\varrho_{\varepsilon}(t)$
satisfy the integral equation \eqref{eq:transf}. Actually, this can be
written as a Liouville (continuity) equation with respect to the
classical Hamiltonian of the Klein-Gordon-Schrödinger system. Proving
uniqueness of solutions of the latter equation implies that the
measure $\tilde\mu_t$ is the push forward of $\mu_0$ (the Wigner
measure at time $t=0$) by the classical flow $\Phi(t,0)$ which is a
well defined continuous map on $\Z$ by Proposition \ref{prop:3}.

One of our concerns is the regularity with respect to time of the curve
$t\mapsto \tilde\mu_{t}$ as a map valued on $\mathscr{P}(\Z)$, the
space of Borel probability measures over $\Z$.  For our purpose, the
most appropriate topology on $\mathscr{P}(\Z)$ is the weakly narrowly
convergence topology which is described below.  Let
$(e_{n})_{n\in\nz}$ be a Hilbert basis of $\Z$. In the following, we
endow $\Z=L^2(\mathds{R}^d)\oplus L^2(\mathds{R}^d)$ by the distance
$d_{\omega}(z_{1},z_{2})=\sqrt{\sum_{n\in\nz}\frac{|\langle
    z_{1}-z_{2},e_{n}\rangle|^{2}}{(1+n)^{2}}}$. It is not difficult
to see that the topology of $(\Z,d_w)$ coincides with the weak
topology on bounded sets. We say that a sequence $(\mu_{n})_{n\in
  \mathds{N}}$ in $\mathscr{P}(\Z)$ weakly narrowly converges to
$\mu\in \mathscr{P}(\Z)$ if
\begin{equation*}
  \forall f\in \mathscr{C}_{b}(\Z,d_{w}), \quad\lim_{n\to \infty} \int_{\Z}  f(z) \, d\mu_{n}=\int_{\Z}  f(z) \, d\mu\,,
\end{equation*}
where $\mathscr{C}_{b}(\Z,d_{w})$ denotes the space of all bounded
continuous real-valued functions on $(\Z,d_{w})$. In practice, it is
more convenient to use cylindrical functions in order to check weak
narrow continuity properties. We recall that a function $f:\Z\to
\mathds{R}$ is in the cylindrical Schwartz space
$\mathcal{S}_{cyl}(\Z)$ if there exists a finite rank orthogonal
projection $\wp$ on $\Z$ and a function $g:\wp\Z\to \mathds{R}$ in the
Schwartz space $\mathcal{S}(\wp\Z)$ such that
\begin{equation*}
  \forall z\in \Z, \quad f(z)=g(\wp z) \,.
\end{equation*}
In the same way, if $g\in \mathcal{C}_0^{\infty}(\wp \Z)$ we can
define the space of smooth cylindrical functions of compact support
$\mathcal{C}_{0,cyl}^{\infty}(\Z)$. We caution the reader that neither
$\mathcal{S}_{cyl}(\Z)$ nor $\mathcal{C}_{0,cyl}^{\infty}(\Z)$
possess a vector space structure.  The Fourier transform of $f\in
\mathcal{S}_{cyl}(\Z)$, based on a finite dimensional subspace
$\wp\Z$, is
\begin{equation}
  \label{eq:fourier}
  \mathcal{F}[f](\xi)=\int_{\wp\Z}  e^{-2i\pi {\mathrm Re}\langle
    \xi\,,\, z\rangle_{\Z}} f(z) ~dL_{\wp}(z)\,,
\end{equation}
where $dL_{\wp}(z)$ denotes the Lebesgue measure on $\wp\Z$ and the
inverse formula is
\begin{equation*}
  f(z)=\int_{\wp\Z}  e^{2i\pi {\mathrm Re}\langle
    \xi\,,\, z\rangle_{\Z}} \mathcal{F}[f](\xi) ~dL_{\wp}(\xi)\,.
\end{equation*}

\begin{proposition}
  \label{prop:liouville}
  Assume that \eqref{eq:5} holds and that
  $(\tilde\mu_t)_{t\in\mathds{R}}$ are Wigner measures of the family
  $(\tilde\varrho_{\varepsilon}(t)))_{\varepsilon\in(0,\bar{\varepsilon})}\in\mathcal{S}^1$
  provided by Proposition \ref{prop:13}.  Then the map
  $t\in\mathds{R}\mapsto\tilde{\mu}_{t}$ is weakly narrowly continuous
  and satisfies the transport equation
  \begin{equation}
    \label{eq:Liouville}
    \partial_{t}\tilde{\mu}_t+\nabla^T\left(\mathscr{V}_{t} \tilde{\mu}_{t}\right)=0\,,
  \end{equation}
  in the weak sense,
  \begin{equation}
    \label{eq.weakcont}
    \forall f\in \mathcal{C}^{\infty}_{0,cyl}(\mathds{R}\times\Z)\,,\quad
    \int_{\mathds{R}}\int_{\Z}
    \left(\partial_{t}f+\Re\langle \nabla f, \mathscr{V}_{t}\rangle\right)~d\tilde{\mu}_{t}dt=0\,.
  \end{equation}
\end{proposition}
\begin{proof}
  For any $f\in \mathcal{S}_{cyl}(\Z)$, based on $\wp\Z$ with $\wp$ a
  finite rank orthogonal projection, Fubini's theorem gives
  $$
  \int_{\Z} f(z)~d\tilde{\mu}_{t}(z)=\int_{\wp\Z} \mathcal{F}[f](\xi)
  \, \tilde{\mu}_{t}(e^{2i\pi {\mathrm Re}\langle \xi\,,\,
    z\rangle_{\Z}}) ~dL_{\wp}(z)\,,
  $$
  where $\mathcal{F}$ is the Fourier transform \eqref{eq:fourier}.
  Hence, by the estimate~\eqref{eq:25} (with $\eta=\xi$) and the decay at infinity of
  $\mathcal{F}[f]$ the map $t\mapsto \int_{\Z}
  f(z)~d\tilde{\mu}_{t}(z)$ is continuous for any $f\in
  \mathcal{S}_{cyl}(\Z)$. Now, the bound
  $\int_{\Z} ||z||_{\Z}^{2}~d\tilde{\mu}_{t}(z)\leq C(T_0)$ (proved in
  Proposition~\ref{prop:12}) and \citep[][Lemma~5.1.12-f)]{AGS}
  guaranties the weak narrow continuity of the curve
  $t\mapsto \tilde{\mu}_{t}$.\\
  The transport equation \eqref{eq:Liouville} follows by integrating
  \eqref{eq:transf} against $\mathcal{F} [g](\xi)~dL_{\wp}(z)$ for any
  $g\in \mathcal{C}_{0,cyl}^{\infty}(\Z)$ based on $\wp\Z$. So, we
  obtain
  \begin{equation*}
    \int_{\Z} g(z)~d\tilde{\mu}_{t}(z)=\int_{\Z} g(z)~d\tilde{\mu}_{0}(z)
    +2i\pi\int_{0}^{t}\int_{\wp\Z} \tilde\mu_{s}\left( \Re\langle \xi,
      \mathscr{V}_{s}(z)\rangle \right)\,\mathcal{F}[g](\xi) ~dL_{\wp\Z}(\xi)~ds\,.
  \end{equation*}
  By Fubini's theorem and properties of finite dimensional Fourier
  transform, the identity
  \begin{equation}
    \label{eq:19}
    \int_{\Z} g(z)~d\tilde{\mu}_{t}(z)=\int_{\Z} g(z)~d\tilde{\mu}_{0}(z)
    \rangle+\int_{0}^{t}\int_{\Z}  \Re\langle \nabla g(z),
  \mathscr{V}_{s}(z)\rangle  ~d\tilde\mu_{s}(z)~ds\,,
  \end{equation}
  holds true with $\nabla g(z)$ the differential of
  $g:\Z\to\mathds{R}$ (here $\Z$ is considered as a real Hilbert space
  with the scalar product $\Re\langle\cdot, \cdot\rangle$). We observe
  that for any $g\in \mathcal{S}_{cyl}(\Z)$ the r.h.s. of
  \eqref{eq:19} is $\mathcal{C}^{1}(\mathds{R})$. Hence, a
  differentiation with respect to $t$ gives
  $$
  \partial_{t} \left( \int_{\Z} g(z)~d\tilde{\mu}_{t}(z) \right)-\int_{\Z}\Re\langle \nabla g(z), \mathscr{V}_{t}(z)\rangle ~d\tilde\mu_{t}(z)=0\,.
  $$
  Thus, multiplying the above relation by $\varphi(t)$\,, with
  $\varphi\in \mathcal{C}^{\infty}_{0}(\mathds{R},\mathds{R})$\,, and
  integrating by part proves \eqref{eq.weakcont} for
  $f(t,z)=\varphi(t)g(z)$\,. We conclude by observing that any
  $f\in\mathcal{C}^{\infty}_{0,cyl}(\mathds{R}\times\Z)$,
  $f(t,z)=g(t,\wp z)$ with
  $g\in\mathcal{C}^{\infty}_{0}(\mathds{R}\times\wp\Z) $ can be
  approximated by a sequence $\bigl(g_{n}(\wp \,\cdot\,,\,
  \cdot\,)\bigr)_{n\in\mathds{N}}$ in
  $\mathcal{C}^{\infty}_{0}(\mathds{R})\stackrel{alg}{\otimes} \mathcal{C}^{\infty}_{0}(\wp\Z)$.
\end{proof}

\begin{proposition}
  \label{prop:regdata}
  Assume that \eqref{eq:5} holds. Let
  $(\varrho_{\varepsilon})_{\varepsilon\in(0,\bar{\varepsilon})}\in\cap_{\delta>0}
  \mathcal{T}^{\delta}\cap \mathcal{S}^1$ and admits a unique Wigner
  measure $\mu_{0}$. Then for any time $t\in \mathds{R}$\,, the family
  $(\varrho_{\varepsilon}(t)=e^{-i\frac{t}{\varepsilon}H_{\varepsilon}}\varrho_{\varepsilon}e^{i\frac{t}{\varepsilon}H_{\varepsilon}})_{\varepsilon\in
    (0,\bar \varepsilon)}$ admits a unique Wigner measure
  $\mu_{t}=\Phi(t,0)_{\#}\mu_{0}$\,, where $\Phi$ is the flow of the
  Klein-Gordon-Schrödinger system defined on $\Z$\, by Proposition
  \ref{prop:3}.
\end{proposition}
\begin{proof}
  Proposition \ref{prop:13} and Proposition \ref{prop:liouville} say
  that for any sequence $ (\varepsilon_n)_{n\in\mathds{N}}\in
  (0,\bar{\varepsilon})$ such that $\lim_{n\to\infty}
  \varepsilon_n=0$; there exists a subsequence $
  (\varepsilon_{n_k})_{k\in\mathds{N}}$, and a family of Wigner
  measures $(\tilde\mu_t)_{t\in\mathds{R}}$ of
  $(\tilde{\varrho}_{\varepsilon})_{\varepsilon\in(0,\bar{\varepsilon})}$
  which are Borel probability measures on $\mathscr{Z}$ satisfying the
  transport equation \eqref{eq:Liouville}-\eqref{eq.weakcont} for all
  $ t\in\mathds{R}$ with initial datum $\mu_{0}$ a time $t=0$.  Now,
  we apply \citep[][Proposition C.8]{2011arXiv1111.5918A} in order to conclude that
  such transport equation \eqref{eq:Liouville} admits a unique
  solution given by
  $$
  \Phi_0(t)_{\#}\tilde\mu_{t}=\Phi(t,0)_{\#}\mu_{0}\text{ i.e. } \mu_{t}=\Phi(t,0)_{\#}\mu_{0}\,.
  $$
  The assumptions to be checked are:\\
  (i) For all $T>0$,
  $$
  \int_{-T}^{T} \left(\int_{\Z}||\mathscr{V}_{t} (z)||_{\Z} ^{2}~d\tilde\mu_t(z)\right)^{1/2} dt<\infty\,.
  $$
  This holds true by~\eqref{eq:22} and the a priori estimate~\eqref{eq:24}.\\
  (ii) The map $t\in\mathds{R}\mapsto \tilde\mu$ is continuous with
  respect to the Wasserstein distance $W_{2}$. Indeed,
  \citep[][Proposition C1]{2011arXiv1111.5918A} shows that a weakly
  narrowly continuous curve satisfying a transport equation with a
  Borel velocity field satisfying (i) is continuous with respect to
  the Wasserstein distance.
\end{proof}

\subsection{Propagation for general states}
\label{se.gendata}
The extension of Proposition \ref{prop:regdata} to general states
$(\varrho_{\varepsilon})_{\varepsilon\in(0,\bar\varepsilon)}$
satisfying the assumption \eqref{eq:0} of Theorem \ref{main.th.1}
follows by a general approximation argument introduced in \cite{AmNi3}
and briefly sketched below. We recall that $S=H_0+T$ and
$T=N_1^2+N_2+1$.  Suppose that for some $\delta>0$ there exists
$C_\delta>0$ such that
\begin{equation}
  \label{eq:20}
  \forall  \varepsilon\in(0,\bar\varepsilon), \quad \lvert S^{\delta/4}\varrho_{\varepsilon} S^{\delta/4}\rvert_{\mathcal{L}^1(\mathscr{H})}^{}\leq C_{\delta}\,.
\end{equation}
Let $\chi\in\mathscr{C}^{\infty}_{0}(\mathds{R})$ such that $0\leq
\chi\leq 1$\,, $\chi\equiv 1$ in a neighbourhood of $0$\, and
$\chi_R(x)=\chi(\frac x R)$. Then the family of normal states
$$
\varrho_{\varepsilon,R}=\frac{\chi_R(S)\varrho_{\varepsilon}\chi_R(S)}{\Tr\left[\chi_R(S)\varrho_{\varepsilon}\chi_R(S)\right]}
$$
approximate $\varrho_{\varepsilon}$ as $R\to \infty$.  Notice that
$\varrho_{\varepsilon,R}$ is well defined for $R$ sufficiently large
for all $\varepsilon\in(0,\bar\varepsilon)$.  Actually, thanks to the
assumption \eqref{eq:0},
$$
\lvert\varrho_{\varepsilon}(t)-\varrho_{\varepsilon,R}(t)\rvert_{\mathcal{L}^1(\mathscr{H})}^{}= \lvert\varrho_{\varepsilon}-\varrho_{\varepsilon,R}\rvert_{\mathcal{L}^{1}(\mathcal{H})} \leq \nu(R)
$$
where
$\varrho_{\varepsilon,R}(t)=e^{-i\frac{t}{\varepsilon}H_{\varepsilon}}\varrho_{\varepsilon,R}\,
e^{i\frac{t}{\varepsilon}H_{\varepsilon}}$ and $\nu(R)$ is independent
of $\varepsilon$ with $\lim_{R\to \infty}\nu(R)=0$\,. Now, it is easy
to see that for any $R\in(0,\infty)$ the family of states
$(\varrho_{\varepsilon,R})_{\varepsilon\in(0,\bar\varepsilon)}$
satisfies the assumptions of Proposition~\ref{prop:regdata} except the
uniqueness of the Wigner measure at time $t=0$\,. However, up to
extracting a sequence which a priori depends on $R$, we can suppose
that $\mathscr{M}(\varrho_{\varepsilon_{n},R},
n\in\nz)=\left\{\mu_{0,R}\right\}\,$ and
$\mathscr{M}(\varrho_{\varepsilon_{n}},
n\in\nz)=\left\{\mu_{0}\right\}\,$.  Thus, we obtain
$$
\forall t\in\mathds{R}\,,\quad \mathcal{M}(\varrho_{\varepsilon_{n}}(t), n\in\nz)=\left\{\Phi(t,0)_{\#}\mu_{0,R}\right\}\,.
$$
For each $t\in\mathds{R}$, we can again extract a subsequence, which may
depend on $t$, such that
$$
\mathcal{M}(\varrho_{\varepsilon_{n}}(t), n\in\nz)=\left\{\mu_t\right\}\,.
$$
Now, \cite[Proposition~2.10 ]{AmNi3} implies
\begin{eqnarray*}
  &&\int_{\Z}|\mu_{t}-\Phi(t,0)_{\#}\mu_{0,R}|\leq \liminf_{n\to\infty}
  \lvert\varrho_{\varepsilon_n}(t)-\varrho_{\varepsilon_n,R}(t)
  \rvert_{\mathcal{L}^{1}(\mathcal{H})}\leq\nu(R)\,, \quad\mbox{ and }\\
  &&\int_{\Z}|\mu_{0}-\mu_{0,R}|\leq \liminf_{n\to\infty}
  \lvert\varrho_{\varepsilon_n}-\varrho_{\varepsilon_n,R}
  \rvert_{\mathcal{L}^{1}(\mathcal{H})}\leq\nu(R)\,,
\end{eqnarray*}
where the left hand side denotes the total variation of the signed measures
$\mu_t-\Phi(t,0)_{\#}\mu_{0,R}$ and $\mu_0-\mu_{0,R}$. Therefore, we
obtain
$$
\int_{\Z}|\mu_t-\Phi(t,0)_{\#}\mu_{0}|\leq \int_{\Z}|\mu_t-\Phi(t,0)_{\#}\mu_{0,R}|+\int_{\Z} |\mu_{0,R}-\mu_{0}| \leq 2\nu(R)\,,
$$
since the total variation of
$\Phi(t,0)_{\#}\mu_{0,R}-\Phi(t,0)_{\#}\mu_{0}$ and
$\mu_{0,R}-\mu_{0}$ are equal.  Letting $R\to \infty$ implies
$\mu_t=\Phi(t,0)_{\#}\mu_{0}$. Thus, the argument above shows
$$
\mu_t\in\mathcal{M}(\varrho_{\varepsilon}(t), \varepsilon\in (0,\bar\varepsilon)) \Leftrightarrow \left(\mu_t=\Phi(t,0)_{\#}\mu_{0}, \mu_0\in\mathcal{M}(\varrho_{\varepsilon}, \varepsilon\in (0,\bar\varepsilon))\right)\,.
$$
It is easy to see that the assumption \eqref{eq:0} implies
\eqref{eq:20}.  This ends the proof of Theorem~\ref{main.th.1}.

\section{Ground state energy limit}
\label{sec:gdstate}
In this section we give the proof of Theorem \ref{main.th.2}. We
recall that we assume \eqref{eq:5}, $m_0>0$ and suppose that $V$ is a
confining potential (i.e.: $\lim_{|x|\to \infty}V(x)=+\infty$). The
classical energy functional related to the Klein-Gordon-Schrödinger
system is given by $ h(z)=h_{0}(z)+h_{I}(z)$ where
$$
h_{0}(z)=\langle z_1,(-\frac{\Delta}{2M}+V) z_1\rangle+\langle z_2,\omega(k) z_2\rangle\,, \quad z=z_{1}\oplus z_{2}\in D((\frac{-\Delta}{2M}+V)^{1/2})\oplus D(\omega^{1/2})\,,
$$
is the quadratic positive part while $h_{I}(z)$ is the nonlinear
regular one given by
$$
h_{I}(z)= \int_{\mathds{R}^{2d}}^{}\frac{\chi(k)}{\sqrt{\omega(k)}} |z_1|^2(x) \bigl(\bar{z}_2(k) e^{-ik\cdot x}+z_2(k)e^{ik\cdot x}\bigr) dkdx\;, \quad z=z_{1}\oplus z_{2}\in\Z.
$$
Actually, the simple inequality $|h_{I}(z)|\leq 2 ||z_{1}||_2^{2}
||\frac{\chi}{\sqrt{\omega}}||_2 ||z_{2}||_2 $ holds true as well as
the scaling $h(\lambda z)=\lambda^{2} h_{0}(z)+\lambda^{3} h_{I}(z)$
for any $\lambda\in\mathds{R}$. Therefore, the functional $h$ is
unbounded from below whenever $\chi$ is different from zero.  However,
the Nelson Hamiltonian preserves the number of nucleons and the ground
state energy of $H_{|\mathscr{H}_n}$ is bounded from below (here
$\mathscr{H}_n=L_{s}^{2}(\mathds{R}^{dn})\otimes\Gamma_s(L^2(\mathds{R}^{d}))$
and $L_{s}^{2}(\mathds{R}^{dn})$ is the space of symmetric square
integrable functions).  This means classically that the
Klein-Gordon-Schrödinger system preserves the $L^{2}$ norm of
$z_{1}$ and $h$ is bounded from below under the constraint
$||z_{1}||_2=\lambda$ with $\lambda$ fixed.
\begin{lemma}
  Assume \eqref{eq:5} and $m_0>0$. Then, for any $\lambda>0$,
  \begin{equation*}
    \inf_{||z_1||_2=\lambda} h(z_1\oplus z_2)>-\infty\,.
  \end{equation*}
\end{lemma}
\begin{proof}
  A phase space translation shows for $z=z_{1}\oplus z_{2}$ such that
  $||z_{1}||_2=\lambda$ that the energy functional can be written as
  \begin{equation*}
    h(z)=\langle z_1,(-\frac{\Delta}{2M}+V) z_1\rangle+\int_{\mathds{R}^{d}} \langle\, \frac{z_2}{\lambda}+\lambda \frac{e^{-ikx}}{\omega^{3/2}}\chi,\omega(k) \,  \bigl(\frac{z_2}{\lambda}+\lambda \frac{e^{-ikx}}{\omega^{3/2}} \chi\bigr)\,\rangle \, |z_{1}|^{2}(x) dx- \lambda^{4}  ||\frac{\chi}{\omega}||_2^{2}\,.
  \end{equation*}
  Observe that $\frac{z_2}{\lambda}+\lambda
  \frac{e^{-ikx}}{\omega^{3/2}}\chi$ belongs to
  $\omega^{-1/2}L^{2}(\mathds{R}^{d})$, so that all the terms make
  sense.  Hence, the quantitative bound $h(z)\geq - \lambda^{4}
  ||\frac{\chi}{\omega}||_2^{2}$ holds true.
\end{proof}

\subsection{Upper bound}
The upper bound is very simple to prove. It follows by an appropriate
choice of trial functions (coherent type states) for the quantum
energy.
\begin{lemma}
  \label{lem:upper}
  Let $\lambda>0$. Then for any $\varepsilon\in(0,\bar\varepsilon)$
  and $n\in\mathds{N}$ such that $n\varepsilon=\lambda^{2}$,
  \begin{equation}
    \displaystyle
    \inf \sigma(H_{|\mathscr{H}_{n}}) \leq \displaystyle\inf_{||z_{1}||_2=\lambda} h(z_{1}\oplus z_{2})\,.
  \end{equation}
\end{lemma}
\begin{proof}
  Take for $\lambda>0$, $z_{1}\in C_{0}^{\infty}(\mathds{R}^{d})$ such
  that $||z_{1}||_2=\lambda$ and $z_{2}\in D(\omega)$, the coherent
  vector
  $$
  C(z_{1}\oplus z_{2})=(\frac{z_{1}}{\lambda})^{\otimes n} \otimes W(\frac{\sqrt{2}}{i\varepsilon} z_{2}) \Omega\,,
  $$
  with $\Omega=(1,0\cdots)$ the vacuum vector of the Fock space
  $\Gamma_s(L^2(\mathds{R}^d))$. It is easy to check that
  $C(z_{1}\oplus z_{2})$ belongs to the domain
  $D(H_{|\mathscr{H}_{n}})=D(H_{0|\mathscr{H}_{n}})$ since
  $(\frac{z_{1}}{\lambda})^{\otimes n} $ is in
  $D(d\Gamma(-\frac{\Delta}{2M}+V))$ and $
  W(\frac{\sqrt{2}}{i\varepsilon} z_{2}) \Omega$ is in
  $D(d\Gamma(\omega))$. Using the fact $n\varepsilon=\lambda^{2}$, an
  explicit computation yields
  $$
  \langle C(z_{1}\oplus z_{2}), H_{|\mathscr{H}_{n}} C(z_{1}\oplus z_{2})\rangle= h(z_{1}\oplus z_{2})\,.
  $$
\end{proof}

\subsection{Lower bound}
The lower bound proof is more elaborated and uses an a priori
information on Wigner measures of minimizing sequences.  It is
convenient to work with
\begin{equation*}
  \mathscr{D}= C_0^\infty(\mathds{R}^{nd})\otimes_{alg}\left(\mathscr{F}\cap D(d\Gamma(\omega))\right)\,,
\end{equation*}
where $\mathscr{F}$ denotes the dense subspace of finite particles
vectors of the Fock space $\Gamma_s(L^2(\rz^d))$.
\begin{lemma}
  \label{lem:mini} Let $\lambda>0$.  There exists a normalized
  minimizing sequence $(\Psi^{(n)})_{n\in\nz}$ in $\mathscr{D}$, such
  that for all $\varepsilon\in(0,\bar\varepsilon)$,
  $n\varepsilon=\lambda$,
  \begin{equation}
    \label{eq.20}
    \langle \Psi^{(n)}, H_{|\mathscr{H}_{n}} \Psi^{(n)}\rangle \leq \frac1 n+\inf\sigma(H_{|\mathscr{H}_{n}})\,.
  \end{equation}
\end{lemma}
\begin{proof}
  Remember that the Kato-Rellich theorem applies for
  $H_{|\mathscr{H}_n}$. Therefore
  $D(H_{|\mathscr{H}_n})=D(H_{0|\mathscr{H}_n})$ and since
  $\mathscr{D}$ is a core for $H_{0|\mathscr{H}_n}$ then it is also a
  core for $H_{|\mathscr{H}_n}$. Thus, one can construct a normalized
  sequence in $\mathscr{D}$ satisfying the inequality \eqref{eq.20}
  since
  $$
  \inf\sigma(H_{|\mathscr{H}_{n}})=\inf_{||\Psi^{(n)}||=1, \Psi^{(n)}\in \mathscr{D}} \langle \Psi^{(n)}, H_{|\mathscr{H}_n} \Psi^{(n)}\rangle\,.
  $$
\end{proof}

\begin{lemma}
  \label{loclimi}
  Let $(\Psi^{(n)})_{n\in\nz}$ be a minimizing sequence as in Lemma
  \ref{lem:mini}.  We can assume that $(\Psi^{(n)})_{n\in\nz}$ has a
  unique Wigner measure $\mu$. Then for any $R>0$,
  \begin{equation*}
    \lim_{n\to\infty} \langle \Psi^{(n)}, d\Gamma( 1_{|x|\leq R}) \otimes 1\,\Psi^{(n)}\rangle=\int_{\Z} \langle z_{1}, 1_{|x|\leq R} \, z_{1}\rangle \, d\mu(z)\,.
  \end{equation*}
\end{lemma}
\begin{proof}
  Proposition \ref{prop:10} ensures the existence of Wigner measures
  for $(\Psi^{(n)})_{n\in\nz}$ since
  $$
  \langle \Psi^{(n)}, N\,\Psi^{(n)}\rangle \leq \lambda^2+\langle \Psi^{(n)}, H_{0|\mathscr{H}_n}\,\Psi^{(n)}\rangle\,,
  $$
  and the right hand side is uniformly bounded with respect to
  $n\in\mathds{N}$.
  Moreover, by extracting a subsequence we can always  assume that $(\Psi^{(n)})_{n\in\nz}$ has a unique Wigner measure.\\
  Let $\tilde\chi\in C_0^\infty(\rz)$ such that $0\leq
  \tilde\chi(x)\leq 1$, $\tilde\chi(x)=1$ if $|x|\leq 1$ and
  $\tilde\chi(x)=0$ if $|x|\geq 2$. Let
  $\tilde\chi_\kappa(x)=\tilde\chi(\frac{x}{\kappa})$, for $\kappa>0$.
  \begin{eqnarray}
    \left| \lambda^2\langle \Psi^{(n)}, 1_{|x_1|\leq R} \,\Psi^{(n)}\rangle-\int_{\Z}\langle z_1, 1_{|x|\leq R} \, z_1\rangle \, d\mu(z)\right| &&\leq\left|\lambda^2\langle \Psi^{(n)}, 1_{|x_1|\leq R} \,[1-\tilde\chi_\kappa(D^2_{x_1})] \,\Psi^{(n)}\rangle\right| \\
    && +\left|\lambda^2\langle \Psi^{(n)}, 1_{|x_1|\leq R} \,\tilde\chi_\kappa(D^2_{x_1})\,\Psi^{(n)}\right.\rangle\\
    &&\left.-\int_{\Z}\langle z_1, 1_{|x|\leq R} \tilde\chi_\kappa(D_x^2) \, z_1\rangle \, d\mu(z)\right| \\
    && +\left| \int_{\Z}\langle z_1, 1_{|x|\leq R} [\tilde\chi_\kappa(D_x^2)-1] \, z_1\rangle \, d\mu(z)\right|
  \end{eqnarray}
  The first term in right hand side can be estimated by
  \begin{eqnarray}
    \label{eq.1}
    \left|\langle \Psi^{(n)}, 1_{|x_1|\leq R} \,[1-\tilde\chi_\kappa(D^2_{x_1})] \,\Psi^{(n)}\rangle\right|&\leq&||(1+D^2_{x_1})^{\frac{1}{2}} \,\Psi^{(n)}|| \; ||[1-\tilde\chi_\kappa(D^2_{x})] (1+D^2_x)^{-\frac{1}{2}}||.
  \end{eqnarray}
  So, the left hand side of \eqref{eq.1} tend to zero, uniformly with respect to
  $R>0$, when $\kappa\to\infty$. Now, observe that the operator
  $1_{|x|\leq R} \,\tilde\chi_\kappa(D^2_x)$ is compact. Then by
  Proposition \ref{prop:11} and Remark \ref{rem:2}, we get for all
  $\kappa>0$
  \begin{eqnarray*}
    \lim_{n\to\infty}  \lambda^2\langle \Psi^{(n)}, 1_{|x_1|\leq R} \,\tilde\chi_\kappa(D_{x_1}^2)\,\Psi^{(n)}\rangle&=&\lim_{n\to\infty}\langle \Psi^{(n)}, d\Gamma(1_{|x|\leq R} \,\tilde\chi_\kappa(D_{x_1}^2)) \otimes 1\,\Psi^{(n)}\rangle\\ &=&\int_{\Z}\langle z_1, 1_{|x|\leq R} \,
    \tilde\chi_\kappa(D_x^2)\, z_1\rangle \, d\mu(z)\,.
  \end{eqnarray*}
  Since $\tilde\chi_\kappa(D_x^2)$ converges strongly to $1$, we see
  by dominated convergence theorem that
  $$
  \lim_{\kappa\to\infty} \int_{\mathscr{Z}} \langle z_1, 1_{|x|\leq R} [\tilde\chi_\kappa(D_x^2)-1] \, z_1\rangle \, d\mu(z)=0\,.
  $$
  Hence an $\eta/3$-argument proves the limit.
\end{proof}

\begin{lemma}
  \label{lem:gr1}
  Let $\lambda>0$ and $(\Psi^{(n)})_{n\in\nz}$ be a minimizing
  sequence as in Lemma \ref{lem:mini}. Then there exists $C>0$ such
  that for any $R>0$ and any $n\in\mathds{N}$,
  $n\varepsilon=\lambda^2$,
  $$
  \langle \Psi^{(n)}, d\Gamma(1_{|x|\leq R})\otimes 1 \, \Psi^{(n)}\rangle \geq \lambda^2- \frac{C}{C(R)}
  $$
  with $C(R)=\inf\{V(x), |x|>R\}$.
\end{lemma}
\begin{proof}
  Remark that
  \begin{eqnarray*}
    \lambda^2&=&\langle \Psi^{(n)}, N_1 \, \Psi^{(n)}\rangle\\ &=&\langle \Psi^{(n)}, d\Gamma(1_{|x|\leq R})\otimes 1\, \Psi^{(n)}\rangle+\langle \Psi^{(n)}, d\Gamma(1_{|x|> R})\otimes 1 \, \Psi^{(n)}\rangle.
  \end{eqnarray*}
  Using Lemma \ref{lemma:2}, one can see that $H_{I|\mathscr{H}_n}$ is
  bounded by $H_{02}^{1/2}$ uniformly in $n\in\mathds{N}$, in the
  operator sense. Hence, $(\langle \Psi^{(n)}, H_{02}+H_I\,
  \Psi^{(n)}\rangle)_{n\in\mathds{N}}$ is bounded from below and since
  $\Psi^{(n)}$ is a minimizing sequence there exists $C>0$ such that
  $\langle \Psi^{(n)}, d\Gamma(V(x))\otimes 1 \, \Psi^{(n)}\rangle\leq
  C$. Using the inequality $d\Gamma(V(x))\geq C(R)
  d\Gamma(1_{|x|>R})$, one obtains
  \begin{eqnarray*}
    \langle \Psi^{(n)}, d\Gamma(1_{|x|\leq R})\otimes 1\, \Psi^{(n)}\rangle&=&\lambda^2-\langle \Psi^{(n)}, d\Gamma(1_{|x|> R})\otimes 1 \, \Psi^{(n)}\rangle\\ &\geq& \lambda^2-\frac{C}{C(R)}\,.
  \end{eqnarray*}
\end{proof}

\begin{lemma}
  \label{lem:conc}
  Let $(\Psi^{(n)})_{n\in\nz}$ be a minimizing sequence as in Lemma
  \ref{lem:mini}. Then any Wigner measure $\mu\in
  \mathscr{M}(\psi^{(n)},n\in\mathds{N}, \varepsilon n=\lambda^2)$ is supported on
  $S(0,\lambda)\times L^2(\rz^d)$ where $S(0,\lambda)$ is the sphere
  of $L^2(\rz^d)$ of radius $\lambda$.
\end{lemma}
\begin{proof}
  Observe that $\langle \Psi^{(n)}, N_1^k \,
  \Psi^{(n)}\rangle=\lambda^{2k}$ for all $k\in\mathds{N}$. Hence,
  \citep[][theorem 6.2]{AmNi1} shows that $\mu$ is supported on
  $B(0,\lambda)\times L^2(\rz^d)$ where $B(0,\lambda)$ is the ball in
  $L^2(\rz^d)$ of radius $\lambda$ centered at the origin.  Using
  Lemma \ref{lem:gr1} and Lemma \ref{loclimi}, we obtain for any $R>0$
  \begin{equation*}
    \int_{B(0,\lambda)\times L^2(\rz^d)} ||z_1||_2^2 \, d\mu(z)\geq \int_{B(0,\lambda)\times L^2(\rz^d)} \langle z_1, 1_{|x|\leq R} z_1\rangle  \, d\mu(z) \geq \lambda^2-\frac{C}{C(R)}.
  \end{equation*}
  Recall that $\lim_{|x|\to \infty} V(x)=+\infty$ so that $
  C(R)\to\infty$ when ${R\to\infty}$.
\end{proof}

\begin{lemma}
  \label{lem:lower}
  For any $\lambda>0$,
  $$
  \liminf_{n\to\infty, n\varepsilon=\lambda^2} \inf\sigma(H_{|\mathscr{H}_n}) \geq \inf_{||z_1||_2=\lambda} h(z_1\oplus z_2)\,.
  $$
\end{lemma}
\begin{proof}
  Let $(\Psi^{(n)})_{n\in\mathds{N}}$ be a minimizing sequence as in
  Lemma \ref{lem:mini}. Recall that the annihilation distribution
  $a(k), k\in\mathds{R}^d$ is a well defined operator
  $a(.):\mathscr{F}\to
  L^2\left(\mathds{R}^d,\Gamma_s(L^2(\mathds{R}^d))\right)$. A direct
  computation, using symmetry and Fubini, gives
  \begin{eqnarray*}
    \theta(n):=\langle \Psi^{(n)}, H_{02}+H_{I|\mathscr{H}_{n}}\Psi^{(n)}\rangle &=& \int_{\mathds{R}^d} ||a(k) \Psi^{(n)}||_{\Gamma_s(L^2(\mathds{R}^d))}^2 \, \omega(k)\, dk\\
    &&+\lambda^2 \int_{\mathds{R}^d\times \mathds{R}^{dn}} \frac{e^{ik x_1}}{\sqrt{\omega(k)}} \chi(k) \, \left(\langle\Psi^{(n)}, a(k) \Psi^{(n)}\rangle_{\Gamma_s(L^2(\mathds{R}^d))}+hc\right) dk dx\,.
  \end{eqnarray*}
  Therefore, we can write
  \begin{eqnarray}
    \label{eq.26}
    \theta(n)=\int_{\mathds{R}^d\times \mathds{R}^{dn}} \omega(k) ||\left(a(k)+\lambda^2\frac{e^{-ik x_1}}{\omega(k)^{3/2}}\chi(k)\right) \Psi^{(n)}||^2_{\Gamma_s(L^2(\mathds{R}^d))}\, dx dk-\lambda^4 ||\frac{\chi}{\omega}||^2_2\,.
  \end{eqnarray}
  Taking any cut off function $0\leq\tilde\chi\leq 1$. Let $Q_\kappa$
  be a sequence of positive finite rank operators such that $0\leq
  Q_\kappa\leq \tilde\chi\omega$ and $
  Q_\kappa\overset{w}{\rightarrow}\tilde\chi\omega$.  Let
  $\{e_\alpha\}_{\alpha\in\mathds{N}}$ be an O.N.B of $L^2(\rz^d)$ so
  that $Q_\kappa= \sum_{\alpha=0}^r t_\alpha |e_\alpha\rangle\langle
  e_\alpha|$ (for simplicity the dependence on $\kappa$ is
  omitted). Expanding all the integrals and sums in \eqref{eq.26},
  then using $Q_\kappa\leq \tilde\chi\omega$, one proves
  \begin{eqnarray*}
    \theta(n)&\geq& \langle \Psi^{(n)}, 1\otimes d\Gamma(Q_\kappa)\Psi^{(n)}\rangle+ \sum_{\alpha=0}^r t_\alpha\langle a(e_\alpha) \Psi^{(n)}, d\Gamma(\widehat{\frac{\chi \bar e_\alpha}{\omega^{3/2}}})\otimes 1\Psi^{(n)}\rangle+hc\\
    &&+\lambda^2\sum_{\alpha=0}^r t_\alpha \langle \Psi^{(n)}, d\Gamma(\big|\widehat{\frac{\chi \bar e_\alpha}{\omega^{3/2}}}\big|^2)\otimes 1\Psi^{(n)}\rangle-\lambda^4 ||\frac{\chi}{\omega}||^2_2\,.
  \end{eqnarray*}
  The right hand side is the expectation value of a Wick operator
  with symbol given by
  \begin{equation*}
    \Theta(z)=\langle z_2,Q_\kappa z_2\rangle+\int_{\mathds{R}^d} \left(\langle z_2, Q_\kappa \frac{\chi e^{-ikx}}{\omega^{3/2}}\rangle+hc\right) |z_1(x)|^2 dx+\lambda^2 \sum_{\alpha=0}^r t_\alpha \langle z_1, \big|\widehat{\frac{\chi \bar e_\alpha}{\omega^{3/2}}}\big|^2 z_1\rangle
    -\lambda^4 ||\frac{\chi}{\omega}||^2_2\,.
  \end{equation*}
  In this symbol some monomials have non "compact kernels" (see the
  discussion in Section \ref{sec:wigner-measures}).  So, using the
  same approximation scheme as in Definition \ref{def:1} and Lemma
  \ref{lemma:9}, we show
  \begin{eqnarray*}
    \theta(n)&\geq& \langle \Psi^{(n)}, 1\otimes d\Gamma(Q_\kappa)\Psi^{(n)}\rangle+ \sum_{\alpha=0}^r t_\alpha\langle a(e_\alpha) \Psi^{(n)}, d\Gamma(\widehat{\frac{\chi \bar e_\alpha}{\omega^{3/2}}}g_m(i\partial_x))\otimes 1\Psi^{(n)}\rangle+hc\\
    &&+\lambda^2\sum_{\alpha=0}^r t_\alpha \langle \Psi^{(n)}, d\Gamma(\big|\widehat{\frac{\chi \bar e_\alpha}{\omega^{3/2}}}\big|^2 g_m(i\partial_x))\otimes 1\;\Psi^{(n)}\rangle-\lambda^4 ||\frac{\chi}{\omega}||^2_2\,+O(m^{-1})\,,
  \end{eqnarray*}
  with an error uniform in $n\in\mathds{N}$. Now, the point is that
  the right hand side is an expectation value of a Wick quantization
  with compact kernel symbol. We can apply the same argument as in
  Proposition \ref{prop:11} and Remark \ref{rem:2}. Therefore, we
  obtain
  $$
  \liminf_{n\to\infty}\theta_n\geq \int_{\mathscr{Z}} \Theta_m(z)\,d\mu(z)\,,
  $$
  where $\mu$ is the Wigner measure of the sequence
  $(\Psi^{(n)})_{n\in\mathds{N}}$ and
  \begin{eqnarray*}
    \Theta_m(z)&=&\langle z_2,Q_\kappa z_2\rangle+\sum_{\alpha=0}^r  t_\alpha \left(\langle z_2, e_\alpha\rangle \langle z_1, \widehat{\frac{\chi \bar e_\alpha}{\omega^{3/2}}}g_m(i\partial_x) z_1\rangle+hc\right) \\
    &&+\lambda^2 \langle z_1, \big|\widehat{\frac{\chi \bar e_\alpha}{\omega^{3/2}}}\big|  g_m(i\partial_x) z_1\rangle-\lambda^4 ||\frac{\chi}{\omega}||^2_2\,.
  \end{eqnarray*}
  We can remove, by dominated convergence, the cut off $g_m$ and let
  $\kappa\to\infty$. So we obtain
  $$
  \liminf_{n\to\infty}\theta_n\geq \int_{\mathscr{Z}} \langle z_2,\omega z_2\rangle+h_I(z)\,d\mu(z)\,,
  $$
  Now, a similar argument of approximation from below gives
  \begin{equation*}
    \langle \Psi^{(n)}, 1\otimes d\Gamma(\frac{-\Delta}{2M}+V) \Psi^{(n)}\rangle \geq \langle \Psi^{(n)}, 1\otimes d\Gamma(\tilde\chi(\frac{-\Delta}{2M}+V)) \Psi^{(n)}\rangle\,,
  \end{equation*}
  where $\tilde\chi$ is a cut off function, $\tilde\chi(x)=x$ on $|x|\leq 1$, so that   $\tilde\chi(\frac{-\Delta}{2M}+V)$ is a compact operator. Applying
  Proposition \ref{prop:11}, we get
  $$
  \liminf_{n\to\infty, n\varepsilon= \lambda^2}\langle \Psi^{(n)}, H_{|\mathscr{H}_{n}} \Psi^{(n)}\rangle \geq \int_{\Z} h(z) d\mu(z).
  $$
  Therefore, we obtain
  \begin{equation*}
    \inf_{||z_1||=\lambda}h(z)\leq \int_{S(0,\lambda)\times L^2(\mathds{R}^d)}  h(z) \, d\mu(z) \leq  \liminf_{n\to\infty, n\varepsilon=\lambda^2} \langle \Psi^{(n)}, H_{|\mathscr{H}_{n}} \Psi^{(n)}\rangle\leq \liminf_{n\to\infty, n\varepsilon=\lambda^2}\inf\sigma(H_{|\mathscr{H}_n})\,,
  \end{equation*}
  since by Lemma \ref{lem:conc}  the Wigner measure $\mu$ is supported on the sphere of radius   $\lambda$.
\end{proof}
Thus, Lemma \ref{lem:lower} and Lemma \ref{lem:upper} imply Theorem
\ref{main.th.2}.

\begin{remark}
It is not difficult to show that the infimum of the classical energy
$h$, under the constraint $||z_1||_2=\lambda$, is actually a minimum.
\end{remark}
\appendix

\section{Estimates on Fock space.}
\label{sec:estim-fock-spac}

We provide some technical results used throughout the paper and proved
here for general Hilbert spaces.
\begin{lemma}
  \label{lemma:4}
  Let $\mathscr{Y}$ be an Hilbert space, $\Gamma_s(\mathscr{Y})$ the
  corresponding symmetric Fock space (with $a^{\#}$, $N$, $W(\xi)$ the
  annihilation/creation, number and Weyl operators respectively).

  Let $y$ be a positive self-adjoint operator on $\mathscr{Y}$ with
  domain $D(y)$; and let $d\Gamma(y)$ be the second quantization of
  $y$, with form domain $D(Y^{1/2})$. Then for all $\xi\in
  D(y^{1/2})$, and $\phi_1,\phi_2\in D(Y^{1/2})$:
  \begin{equation*}
    \langle \phi_1 , W^{*}(\xi) d\Gamma(y) W(\xi)\phi_2 \rangle_{}=\langle \phi_1 , \bigl(d\Gamma(y)+\frac{i\varepsilon}{\sqrt{2}}(a^{*}(y\xi)-a(y\xi))+\frac{\varepsilon^2}{2}\langle \xi ,y\xi \rangle_{\mathscr{Y}} \bigr)\phi_2\rangle_{}\; .
  \end{equation*}
\end{lemma}
\begin{proof}
  Let $\xi\in D(y^{1/2})$ be fixed, let $\phi_1,\phi_2\in
  D(N)$. Furthermore, let $(y_m)_{m\in\mathds{N}}\in
  \mathcal{L}(\mathscr{Y})$ be a sequence of bounded operators that
  converges strongly to $y$ on $D(y)$, with $y_m\leq y$ for all
  $m$. Then we define, for all $\lambda\in \mathds{R}$,
  \begin{equation*}
    M(\lambda):=\langle \phi_1 , W(\lambda\xi) \bigl(d\Gamma(y_m)+\frac{i\lambda\varepsilon}{\sqrt{2}}(a^{*}(y_m\xi)-a(y_m\xi))+\frac{\lambda^2\varepsilon^2}{2}\langle \xi ,y_m\xi \rangle_{\mathscr{Y}} \bigr)W^{*}(\lambda\xi)\phi_2 \rangle_{}\; .
  \end{equation*}
  We remark that for every $\delta \geq 0$ the Weyl operator maps
  $D(N^{\delta})$ into itself. Taking the derivative in $\lambda$, we
  obtain
  \begin{equation*}
    \begin{split}
      \frac{d}{d\lambda}M(\lambda)=\langle W^{*}(\lambda\xi)\phi_1 , i\bigl[\varphi(\lambda\xi)\; ,\; d\Gamma(y_m)+\frac{i\lambda\varepsilon}{\sqrt{2}}(a^{*}(y_m\xi)-a(y_m\xi))\bigr]W^{*}(\lambda\xi)\phi_2 \rangle_{}+\langle W^{*}(\lambda\xi)\phi_1 ,
      \\\bigl(\frac{i\varepsilon}{\sqrt{2}}(a^{*}(y_m\xi)-a(y_m\xi)) +\lambda\varepsilon^2\langle \xi ,y_m\xi \rangle_{\mathscr{Y}}\bigr)W^{*}(\lambda\xi)\phi_2 \rangle_{}=0\; .
    \end{split}
  \end{equation*}
  Hence for all $\phi_1,\phi_2\in D(N)$ we obtain, by $M(0)=M(1)$,
  for all $m\in \mathds{N}$:
  \begin{equation}
    \label{eq:4}
    \langle \phi_1 , W^{*}(\xi)d\Gamma(y_m)W(\xi)\phi_2 \rangle_{}=\langle \phi_1 , \bigl(d\Gamma(y_m)+\frac{i\varepsilon}{\sqrt{2}}(a^{*}(y_m\xi)-a(y_m\xi))+\frac{\varepsilon^2}{2}\langle \xi ,y_m\xi \rangle_{\mathscr{Y}} \bigr)\phi_2\rangle_{}\; .
  \end{equation}
  Choose now $\phi_1=\phi_2=\phi\in D(Y^{1/2})\cap D(N)$. Then
  \begin{equation*}
    \begin{split}
      \langle \phi , W^{*}(\xi)d\Gamma(y_m)W(\xi)\phi \rangle\leq \lVert d\Gamma(y)^{1/2}\phi \rVert_{}^2+\sqrt{2}\varepsilon\lVert y^{1/2}\xi \rVert_{\mathscr{Y}}^{}\lVert d\Gamma(y)^{1/2}\phi \rVert_{}^{}\lVert \phi \rVert_{}^{}+\frac{\varepsilon^2}{2}\lVert y^{1/2}\xi
      \rVert_{\mathscr{Y}}^2\lVert \phi \rVert_{}^{}\; .
    \end{split}
  \end{equation*}
  By monotone convergence theorem, the left hand side converges to
  $\langle \phi , W^{*}(\xi) d\Gamma(y) W(\xi) \phi \rangle$ when
  $m\to\infty$, since $d\Gamma(y)$ is a closed operator. The result
  extends by density to all $\phi\in D(Y^{1/2})$; so the Weyl operator $W$ maps the form
  domain of $d\Gamma(y)$ into itself. Then for all $\phi_1,\phi_2\in
  D(Y^{1/2})\cap D(N)$, we can take the limit $m\to\infty$
  in~\eqref{eq:4}. The result is then extended by density to all
  $\phi_1,\phi_2\in D(Y^{1/2})$.
\end{proof}
\begin{corollary}
  \label{cor:2}
  \begin{enumerate}[i)]
  \item Let $\xi\in D(y)$. Then
    $(d\Gamma(y)+1)^{-1}W(\xi)(d\Gamma(y)+1)\in\mathcal{L}(\Gamma_s(\mathscr{Y}))$. Furthermore,
    there exists $C(\lVert y \xi \rVert_{\mathscr{Y}}^{},\lVert \xi
    \rVert_{\mathscr{Y}}^{})>0$ independent of $\varepsilon$ such
    that:
    \begin{equation*}
      \lvert (d\Gamma(y)+1)^{-1}W(\xi)(d\Gamma(y)+1)\rvert_{\mathcal{L}(\Gamma_s(\mathscr{Y}))}^{}\leq C(\lVert y \xi \rVert_{\mathscr{Y}}^{},\lVert \xi \rVert_{\mathscr{Y}}^{})(1+O(\varepsilon))\; .
    \end{equation*}
  \item Let $y$ be a positive \emph{bounded} operator and let
    $\xi\in \mathscr{Y}$. Then for any $\delta_1>0$ and $\delta_2\in\mathds{R}$,
    $(d\Gamma(y)^{\delta_1}+1)^{-\delta_2}W(\xi)(d\Gamma(y)^{\delta_1}+1)^{\delta_2}\in\mathcal{L}(\Gamma_s(\mathscr{Y}))$. Furthermore,
    there exists a constant $C(\delta_1,\delta_2,\lVert \xi
    \rVert_{\mathscr{Y}}^{},\lvert
    y\rvert_{\mathcal{L}(\mathscr{Y})}^{})>0$ independent of
    $\varepsilon$ such that:
    \begin{equation*}
      \lvert (d\Gamma(y)^{\delta_1}+1)^{-\delta_2}W(\xi)(d\Gamma(y)^{\delta_1}+1)^{\delta_2}\rvert_{\mathcal{L}(\Gamma_s(\mathscr{Y}))}^{}\leq C(\delta_1,\delta_2,\lVert \xi \rVert_{\mathscr{Y}}^{},\lvert y\rvert_{\mathcal{L}(\mathscr{Y})})(1+O(\varepsilon))\; .
    \end{equation*}
  \end{enumerate}
\end{corollary}

The following proposition is a useful adaptation of \citep[][Lemmas
B.4 and B.6]{2011arXiv1111.5918A}:
\begin{proposition}
  \label{prop:9}
  Let $\mathscr{Y}$ be an Hilbert space, $\Gamma_s(\mathscr{Y})$ the
  corresponding symmetric Fock space.

  Let $y_1,y_2$ be two operators on $\mathscr{Y}$ such that
  $(y_2+1)^{-1}y_1\in \mathcal{L}(\mathscr{Y})$. Then
  $(d\Gamma(y_2^{*}y_2+1)+1)^{-1}d\Gamma(y_1) \in
  \mathcal{L}(\Gamma_s(\mathscr{Y}))$, with:
  \begin{equation*}
    \lvert (d\Gamma(y_2^{*}y_2+1)+1)^{-1}d\Gamma(y_1)\rvert_{\mathcal{L}(\Gamma_s(\mathscr{Y}))}^{}\leq (1+\sqrt{2})\lvert (y_2+1)^{-1}y_1 \rvert_{\mathcal{L}(\mathscr{Y})}^{}\; .
  \end{equation*}
\end{proposition}
\begin{proof}
  Let $\phi_1,\phi_2\in D(d\Gamma(y_1))$. Then ($y(j)$ is the operator
  acting on the $j$-th variable):
  \begin{equation*}
    \begin{split}
      \lvert \langle \phi_1 , d\Gamma(y_1)\phi_2 \rangle_{}\rvert_{}^{}\leq \sum_n\lvert \varepsilon\langle \phi_{1n} , \sum_{j=1}^ny_{1}(j)\phi_{2n}\rangle_{}\rvert_{}^{}\leq \sum_n\lvert \varepsilon n\langle \phi_{1n} , (y_{2}(1)+1)(y_{2}(1)+1)^{-1}y_{1}(1)\phi_{2n}\rangle_{}\rvert_{}^{}\\
      \leq \lvert (y_2+1)^{-1}y_1\rvert_{\mathcal{L}(\mathscr{Y})}^{}\sum_n^{}\lVert \phi_{2n} \rVert_{}^{}\bigl(\lVert \varepsilon n\phi_{1n} \rVert_{}^{}+\lVert \varepsilon n y_2(1)\phi_{1n}\rVert_{}^{}\bigr)\; .
    \end{split}
  \end{equation*}
  However, we have that:
  \begin{equation*}
    \begin{split}
      \lVert \varepsilon n y_2(1)\phi_{1n}\rVert^2=\langle \phi_{1n} , \varepsilon^2n^2y_2^{*}(1)y_2(1)\phi_{1n}  \rangle_{}=\langle \phi_{1n} , d\Gamma(1) d\Gamma(y_2^{*}y_2)\phi_{1n}  \rangle\leq \frac{1}{2}\langle \phi_{1n} , \Bigl( \bigl(d\Gamma(1)\bigr)^2\\+
      \bigl(d\Gamma(y_2^{*}y_2)\bigr)^2\Bigr)\phi_{1n}  \rangle_{}\leq \frac{1}{2}\Bigl(\lVert d\Gamma(1)\phi_{1n} \rVert_{}^2+\lVert d\Gamma(y_2^{*}y_2)\phi_{1n} \rVert_{}^2\Bigr)\; .
    \end{split}
  \end{equation*}
  Hence, we obtain for any $\phi_1,\phi_2\in\Gamma_s(\mathscr{Y})$:
  \begin{equation*}
    \begin{split}
      \lvert \langle \phi_1 , (d\Gamma(y_2^{*}y_2+1)+1)^{-1}d\Gamma(y_1)\phi_2 \rangle_{}\rvert_{}^{}\leq (1+\sqrt{2})\lvert (y_2+1)^{-1}y_1\rvert_{\mathcal{L}(\mathscr{Y})}^{}\sum_n\lVert \phi_{1n} \rVert_{}^{}\lVert \phi_{2n} \rVert_{}^{}\\
      \leq (1+\sqrt{2})\lvert (y_2+1)^{-1}y_1\rvert_{\mathcal{L}(\mathscr{Y})}\lVert \phi_{1} \rVert_{}^{}\lVert \phi_{2} \rVert_{}^{}\; .
    \end{split}
  \end{equation*}
\end{proof}

\begin{acknowledgments}
  The second author has been supported by the Centre Henri Lebesgue (programme
  ``Investissements d'avenir'' --- ANR-11-LABX-0020-01).

The article is published in the Journal of Statistical Physics. The final publication is available at \href{http://dx.doi.org/10.1007/s10955-014-1079-7}{Springer}.
\end{acknowledgments}

\end{document}